\documentclass[12pt]{amsart}
\usepackage[utf8]{inputenc}
\usepackage{amsmath}
\usepackage{amssymb}
\usepackage{amscd}
\usepackage{enumerate}
\usepackage[cal=boondoxo]{mathalfa}
\usepackage{tikz-cd}

\newtheorem{thm}{Theorem}[section]
\newtheorem{cor}[thm]{Corollary}
\newtheorem{prop}[thm]{Proposition}
\newtheorem{lemma}[thm]{Lemma}

\theoremstyle{remark}
\newtheorem{remark}[thm]{Remark}
\newtheorem{example}[thm]{Example}
\newtheorem*{notation}{Notation}

\theoremstyle{definition}
\newtheorem{definition}[thm]{Definition}

\numberwithin{equation}{subsection}

\newcommand{\op}{\mathtt{op}}
\newcommand{\co}{\mathtt{co}}
\newcommand{\id}{{\mathtt{Id}}}
\newcommand{\dbar}{\overline{\partial}}

\DeclareMathOperator{\ad}{ad}
\DeclareMathOperator{\Conj}{\mathtt{c}}
\DeclareMathOperator{\Hom}{Hom}
\DeclareMathOperator{\End}{End}
\DeclareMathOperator{\Aut}{Aut}
\DeclareMathOperator{\Der}{Der}

\DeclareMathOperator{\shExt}{\underline{Ext}}
\DeclareMathOperator{\shAut}{\underline{Aut}}
\DeclareMathOperator{\shHom}{\underline{Hom}}
\DeclareMathOperator{\shEnd}{\underline{End}}
\DeclareMathOperator{\shEq}{\underline{Eq}}
\DeclareMathOperator{\DQ}{\mathfrak{DQ}}
\DeclareMathOperator{\DES}{DES}
\DeclareMathOperator{\shDES}{\underline{\DES}}
\DeclareMathOperator{\gpdExt}{\mathcal{Ext}}
\DeclareMathOperator{\shgpdExt}{\underline{\gpdExt}}
\DeclareMathOperator{\gr}{gr}
\DeclareMathOperator{\Gr}{Gr}
\DeclareMathOperator{\rk}{rk}
\DeclareMathOperator{\avg}{avg}
\DeclareMathOperator{\im}{Im}
\DeclareMathOperator{\coker}{coker}
\DeclareMathOperator{\zentrum}{\mathtt{Z}}

\begin{document}
\title{On the classification of symplectic DQ-algebroids}

\author[P.Bressler]{Paul Bressler}
\address{Departamento de Matem\'aticas, Universidad de Los Andes}
\email{paul.bressler@gmail.com}
\author[J.D.Rojas]{Juan Diego Rojas}
\address{Departamento de Matem\'aticas, Universidad de Los Andes} 
\email{juandrojas33@gmail.com}

\begin{abstract}
DQ-algebroids locally defined on a symplectic manifold form a 2-gerbe. By adapting the method of P. Deligne to the setting of DQ-algebroids we show that this 2-gerbe admits a canonical global section, namely that every symplectic manifold admits a canonical DQ-algebroid quantizing the structure sheaf. The construction relies on methods of non-abelian cohomology and local computations in the Weyl algebra. As a corollary we obtain a classification of symplectic DQ-algebroids.
\end{abstract}

\maketitle

%\author{Paul Bressler and Juan Diego Rojas}
%\thanks{The research of the first author was partially financed by Universidad de Los Andes, Facultad de Ciencias, proyecto INV-2018-50-1356.}
%\address{Departamento de Matem\'aticas, Universidad de Los Andes, Bogot\'a, Colombia \\[5pt]
%Departamento de Matem\'aticas, Universidad de Los Andes, Bogot\'a, Colombia}
%
%\title{On the classification of symplectic DQ-algebroids}
%
%\copyrightyear{2021}
%\keywords{DQ-algebroid, gerbe, deformation quantization}
%\amsclass{53D55}
%
%\eaddress{paul.bressler@gmail.com\CR juandrojas33@gmail.com}
%
%\begin{document}
%
%\maketitle

%\begin{abstract}
%DQ-algebroids locally defined on a symplectic manifold form a 2-gerbe. By adapting the method of P. Deligne to the setting of DQ-algebroids we show that this 2-gerbe admits a canonical global section, namely that every symplectic manifold admits a canonical DQ-algebroid quantizing the structure sheaf. The construction relies on methods of non-abelian cohomology and local computations in the Weyl algebra. As a corollary we obtain a classification of symplectic DQ-algebroids.
%\end{abstract}

\section{Introduction}
While deformation quantization was originally developed in the context of (sheaves of) algebras, it became apparent from the work of M. Kashiwara \cite{Kashiwara-contact} and M. Kontsevich \cite{kontsevich} that the broader context of algebroid stacks provides a natural setting for the theory. These developments lead to the introduction of DQ-algebroids by M. Kashiwara and P. Schapira \cite{KS}.

The classical limit of a DQ-algebroid on a manifold $X$ is not the structure sheaf $\mathcal{O}_X$ but a linear version of a $\mathcal{O}_X^\times$-gerbe. The study of deformation quantization of gerbes was initiated in \cite{DQgerbes} where various earlier results on deformations of (sheaves of) associative algebras were generalized. In particular, deformation quantizations of a gerbe on a symplectic manifold were classified using an extension of B.V. Fedosov's approach.

The question of existence and classification of symplectic deformation quantizations in the $C^\infty$ setting was resolved by M. De Wilde and P.B.A. Lecomte \cite{WL1983,WL1985}, and by B.V. Fedosov \cite{fedosov}. An account of their work which uses some tools of nonabelian cohomology was given by P. Deligne in \cite{D}. In the present note we adapt the method of \cite{D} to the problem of classification of symplectic DQ-algebroids.

Suppose that $X$ is a symplectic manifold with the symplectic form denoted by $\omega$. Since all locally defined DQ-algebroids which give rise to $\omega$ are locally equivalent, it follows that the 2-stack $\DQ^\omega_X$ of such is in fact a 2-gerbe. Moreover, for any locally defined DQ-algebroid $\mathcal{C}\in\DQ^\omega_X$ the stack of autoequivalences is equivalent to the gerbe $\mathbb{C}[[t]]^\times[1]$ of $\mathbb{C}[[t]]^\times$-torsors and, therefore, $\DQ^\omega_X$ is a twisted form of (i.e. is locally equivalent to) the 2-gerbe $\mathbb{C}[[t]]^\times[2]$ of $\mathbb{C}[[t]]^\times$-gerbes. The problem of existence of deformation quantization of $(X,\omega)$ may therefore be restated as the problem of showing that the 2-category $\DQ^\omega_X(X)$ is non-empty. 

The main result of the present work, Theorem \ref{thm: canonical quantization} says that this is indeed the case and that, in fact, there is a canonical choice of a quantization $\mathcal{C}_{can}\in\DQ^\omega_X(X)$, while Theorem \ref{thm: classification} gives an answer to the the classification problem in the form of the canonical equivalence between $\DQ^\omega_X$ and $\mathbb{C}[[t]]^\times[2]$ which induces a bijection between the set of equivalence classes or quantizations $\pi_0\DQ^\omega_X(X)$ and $H^2(X; \mathbb{C}[[t]]^\times)$.

Previously, existence of canonical quantization as well classification results had been established in the complex-analytic context by M. Kashiwara, A. d'Agnolo and P. Polesello in terms of sheaves of algebras of microdifferential and WKB operators in \cite{Kashiwara-contact}, \cite{dagnolo-polesello_complex-involutive-submanifolds}, \cite{polesello-classification}, \cite{andrea-masaki}.

The treatment presented here relies on elementary properties of the Moyal-Weyl star-product and makes heavy use of the theory of abelian (higher) torsors and gerbes. To give a uniform treatment of a range of cases which includes plain $C^\infty$ manifolds as well as complex-analytic manifolds, we work in the natural generality of a $C^\infty$ manifold $X$ equipped with an integrable complex distribution $\mathcal{P}$ which satisfies the technical condition \eqref{sufficient condition} (see Appendix \ref{appendix: distributions} for definitions and notation).

The paper is organized as follows. In Section \ref{section: DQ-algebras} we review the basics of symplectic DQ-algebras and introduce dilation equivariance structures (DES) on DQ-algebras following \cite{D}. In Section \ref{section: Symplectic DQ-algebroids} we recall the basic definitions and facts about DQ-algebroids and describe the stack of symplectic DQ-algebroids equipped with DES. Section \ref{section: Classical limits of DES} is devoted to the study of the behavior of the DES under the classical limit map.
In Section \ref{section: Self-duality structures} we introduce the self-duality structures on symplectic DQ-algebroids and describe the stack of symplectic DQ-algebroids equipped with self-duality structures. In Section \ref{section: Compatibility of DES with self-duality structures} we define what it means for a DES and a self-duality structure to be compatible and describe the stack of symplectic DQ-algebroids equipped with compatible DES and self-duality structures. In Section \ref{section: Classification of symplectic DQ-algebroids} we identify the canonical quantization and state the classification result for symplectic DQ-algebroids.

For the reader's convenience we include two appendices whose content, to a large degree, is borrowed from \cite{QCL}. The relevant basic facts on calculus in the presence of an integrable complex distribution are summarized in Appendix \ref{appendix: distributions}. In Appendix \ref{appendix: algebroid stacks} we give a condensed account of the basic theory of abelian (higher) torsors and gerbes and introduce the notation used throughout the main body of the article. A more detailed presentation of the subject may be found in \cite{Breen} and \cite{M}. 

\section{DQ-algebras}\label{section: DQ-algebras}
Throughout the paper $X$ is a $C^\infty$ manifold equipped with an integrable complex distribution $\mathcal{P}$ which satisfies \eqref{sufficient condition} (see Appendix \ref{appendix: distributions} for definitions and notations). We denote by $\mathcal{O}_X$ (respectively, $\mathcal{O}_{X/\mathcal{P}}$) the sheaf of complex valued $C^\infty$ functions (respectively, the subsheaf of $\mathcal{P}$-invariant functions).

In the context of complex manifolds the notion of a DQ-algebra was introduced in \cite{KS}. 

\subsection{Star-products}
A \emph{star-product} on $\mathcal{O}_{X/\mathcal{P}}$ is a map
\begin{equation*}
\mathcal{O}_{X/\mathcal{P}}\otimes_\mathbb{C}\mathcal{O}_{X/\mathcal{P}} \to \mathcal{O}_{X/\mathcal{P}}[[t]] .
\end{equation*}
of the form
\begin{equation}\label{star-product}
f\otimes g \mapsto f\star g = fg + \sum_{i=1}^\infty P_i(f,g)t^i,
\end{equation}
where $P_i$ are bi-differential operators. Such a map admits a unique $\mathbb{C}[[t]]$-bilinear extension
\begin{equation*}
\mathcal{O}_{X/\mathcal{P}}[[t]]\otimes_{\mathbb{C}[[t]]}\mathcal{O}_{X/\mathcal{P}}[[t]] \to \mathcal{O}_{X/\mathcal{P}}[[t]]
\end{equation*}
and the latter is required to define a structure of an associative unital $\mathbb{C}[[t]]$-algebra on $\mathcal{O}_{X/\mathcal{P}}[[t]]$.

%We refer the reader to \cite{KS}, Proposition 2.2.3 for the proof of the following.

\begin{prop}[{\cite[Proposition 2.2.3]{KS}}]\label{prop: iso star-prod diff op}
Let $\star$ and $\star^\prime$ be star-products and let $\varphi\colon (\mathcal{O}_{X/\mathcal{P}}[[t]], \star) \to (\mathcal{O}_{X/\mathcal{P}}[[t]], \star^\prime)$ be a morphism of $\mathbb{C}[[t]]$-algebras. Then, there exists a unique sequence of differential operators $\{R_i \}_{i\geq 0}$ on $X$ such that $R_0 = 1$ and $\varphi(f) = \sum\limits_{i=0}^\infty R_i(f)t^i$ for any $f\in \mathcal{O}_{X/\mathcal{P}}$. In particular, $\varphi$ is an isomorphism.
\end{prop}

\remark
The paper \cite{KS} and, in particular, Proposition 2.2.3 of loc. cit. pertain to the holomorphic context, i.e. the case when $\mathcal{P}$ is a complex structure. However, it is easy to see that the proof as well as the results it is based upon carry over to the case of a general integrable distribution.
\endremark

\subsection{DQ-algebras}
A DQ-algebra is a sheaf of $\mathbb{C}[[t]]$-algebras locally isomorphic to a star-product on $\mathcal{O}_{X/\mathcal{P}}$. For a DQ-algebra $\mathbb{A}$ there is a canonical isomorphism $\mathbb{A}/t\cdot\mathbb{A}\cong \mathcal{O}_{X/\mathcal{P}}$. Therefore, there is a canonical map  (reduction modulo $t$) $\mathbb{A} \xrightarrow{\sigma} \mathcal{O}_{X/\mathcal{P}}$ of $\mathbb{C}[[t]]$-algebras.

A morphism of DQ-algebras is a morphism of sheaves of $\mathbb{C}[[t]]$-algebras.

\subsection{The associated Poisson structure}
Suppose that $\mathbb{A}$ is a DQ-algebra. The composition
\begin{equation*}
\mathbb{A}\otimes\mathbb{A} \xrightarrow{[\cdot,\cdot]} \mathbb{A} \xrightarrow{\sigma} \mathcal{O}_{X/\mathcal{P}}
\end{equation*}
is trivial. Therefore, the commutator $\mathbb{A}\otimes\mathbb{A} \xrightarrow{[\cdot,\cdot]} \mathbb{A}$ takes values in $t\mathbb{A}$. The composition
\begin{equation*}
\mathbb{A}\otimes\mathbb{A} \xrightarrow{[\cdot,\cdot]} t\mathbb{A} \xrightarrow{t^{-1}} \mathbb{A} \xrightarrow{\sigma} \mathcal{O}_{X/\mathcal{P}}
\end{equation*}
factors uniquely as
\begin{equation*}
\mathbb{A}\otimes\mathbb{A} \xrightarrow{\sigma\otimes\sigma} \mathcal{O}_{X/\mathcal{P}}\otimes\mathcal{O}_{X/\mathcal{P}} \xrightarrow{\{\cdot,\cdot\}} \mathcal{O}_{X/\mathcal{P}}
\end{equation*}
The latter map, $\{\cdot,\cdot\}\colon \mathcal{O}_{X/\mathcal{P}}\otimes\mathcal{O}_{X/\mathcal{P}} \to \mathcal{O}_{X/\mathcal{P}}$ is a Poisson bracket on $\mathcal{O}_{X/\mathcal{P}}$, hence corresponds to a bi-vector $\pi\in \Gamma(X;\bigwedge^2\mathcal{T}_{X/\mathcal{P}})$.

\lemma[{\cite[Lemma 2.6]{QCL}}]\label{lemma: isom dq same poisson}
Locally isomorphic DQ-algebras give rise to the same associated Poisson bracket.
\endlemma

\subsection{Symplectic DQ-algebras}
A DQ-algebra is called \emph{symplectic} if the associated Poisson structure is non-degenerate.

\begin{example}\label{example: MW}
Suppose that $\{\cdot,\cdot\}$ is a symplectic Poisson bracket on $\mathcal{O}_{X/\mathcal{P}}$. Thus, $\rk_{\mathcal{O}_{X/\mathcal{P}}}\mathcal{T}_{X/\mathcal{P}}$ is even, i.e. equal to $2n$ for suitable $n$. Let $U\subset X$ be an open subset and $p_i, q_i \in \mathcal{O}_{X/\mathcal{P}}(U)$, $i=1,\ldots,n$, satisfy the canonical relations. We denote by $\dfrac{\partial\ }{\partial p_1},\ldots,\dfrac{\partial\ }{\partial p_n},\allowbreak\dfrac{\partial\ }{\partial q_1},\ldots,\dfrac{\partial\ }{\partial q_n}$ the basis dual to $dp_1,\ldots,dp_n,dq_1,\ldots,dq_n$. The Moyal-Weyl product is given by
\[
f\star g = \left.\exp\left(\dfrac{t}2\sum\limits_{i=1}^n\left(   
\dfrac{\partial\ }{\partial p_i}\otimes\dfrac{\partial\ }{\partial q_i} - \dfrac{\partial\ }{\partial q_i}\otimes\dfrac{\partial\ }{\partial p_i}
\right)\right)f\otimes g\right\vert_\Delta
\]
where restriction to the diagonal $\Delta$ signifies the multiplication map $\mathcal{O}_{X/\mathcal{P}}\otimes_\mathbb{C} \mathcal{O}_{X/\mathcal{P}} \to \mathcal{O}_{X/\mathcal{P}}$.
\end{example}

\begin{prop}[Quantum Darboux Lemma]\label{prop: QDL}
Suppose that $\mathbb{A}$ is a symplectic DQ-algebra. Then, every point $x\in X$ has a neighborhood $x\in U \subset X$ such that for a collection of functions $p_i, q_i \in \mathcal{O}_{X/\mathcal{P}}(U)$, $i = 1,\ldots, \dfrac12\rk\mathcal{T}_{X/\mathcal{P}}$ which satisfy canonical relations there exist sections $\widehat{p}_i, \widehat{q}_i \in \mathbb{A}(U)$ such that $\sigma(\widehat{p}_i) = p_i$, $\sigma(\widehat{q}_i) = q_i$ and and $[\widehat{p}_i, \widehat{p}_j] = [\widehat{q}_i, \widehat{q}_j] = 0$, $[\widehat{p}_i, \widehat{q}_j] = \delta_{ij}$.
\end{prop}
\begin{cor}\label{cor: props of DQ alg}
{\ }
\begin{enumerate}
\item Symplectic DQ-algebras with the same associated Poisson bracket are locally isomorphic.
\item The unit map $\mathbb{C}[[t]] \to \mathbb{A}$ is an isomorphism onto the center.
\item The sequence
\begin{equation}\label{ses derivations are inner}
0 \to \dfrac1t\mathbb{C}[[t]] \to \dfrac1t\mathbb{A} \xrightarrow{\ad} \Der_{\mathbb{C}[[t]]}(\mathbb{A}) \to 0
\end{equation}
is exact.
\item The sequence of groups
\begin{equation}\label{ses automorphisms are inner}
1 \to \exp(\mathbb{C}[[t]]) \to \exp(\mathbb{A}) \to \shAut(\mathbb{A}) \to 1
\end{equation}
is exact.
\end{enumerate}
\end{cor}
In \eqref{ses automorphisms are inner}, $\exp(\mathbb{A})$ is the pro-unipotent group associated to the pro-nilpotent Lie algebra $\mathbb{A}$ equipped with the commutator bracket, and the map $\exp(\mathbb{A}) \to \shAut(\mathbb{A})$ is given by $``\exp"(a) \mapsto \exp(\ad(a))$.

\begin{proof}
Follows from Proposition \ref{prop: QDL} and well-known properties of the Moyal-Weyl star-product.
\end{proof}

\subsection{Dilation equivariance structures}
Suppose that $\mathbb{A}$ is a symplectic DQ-algebra. The inclusion $\mathbb{C}[[t]] \to \mathbb{A}$ is injective onto the center, so that there is a canonical isomorphism $\mathbb{C}[[t]] \cong \zentrum(\mathbb{A})$. The "restriction to the center" gives rise to the map
\[
(\cdot)\vert_{\zentrum(\mathbb{A})} \colon \Der_\mathbb{C}(\mathbb{A}) \to \Der_\mathbb{C}(\mathbb{C}[[t]]) .
\]
In what follows, we denote by $\partial_t$ the derivation $\dfrac{d~}{dt}$.

\lemma\label{lemma: des exist}
Locally on $X$ there exists $D \in \Der_\mathbb{C}(\mathbb{A})$ such that $D\vert_{\zentrum(\mathbb{A})} = t\partial_t$.
\endlemma
\begin{proof}
Since the question is local, it is sufficient to consider the case of the Moyal-Weyl product described in Example \ref{example: MW}. The natural action of the vector field $\dfrac12\sum\limits_{i=1}^n (p_i\dfrac{\partial\ }{\partial p_i} + q_i\dfrac{\partial\ }{\partial q_i})$ on $\mathcal{O}_{X/\mathcal{P}}(U)$ is easily seen to be a derivation of the Moyal-Weyl product. Hence, $D = t\partial_t + \dfrac12\sum\limits_{i=1}^n (p_i\dfrac{\partial\ }{\partial p_i} + q_i\dfrac{\partial\ }{\partial q_i})$ has the required properties.
\end{proof}
Let $\mathfrak{l}$ denote the $\mathbb{C}$-submodule of $\Der_\mathbb{C}(\mathbb{C}[[t]])$ locally generated by the derivation $t\partial_t$. We define a subsheaf $\mathfrak{L}(\mathbb{A})$ of $\Der_\mathbb{C}(\mathbb{A})$ by the pull-back square
\[
\begin{CD}
\mathfrak{L}(\mathbb{A}) @>>> \mathfrak{l} \\
@VVV @VVV \\
\Der_\mathbb{C}(\mathbb{A}) @>{(\cdot)\vert_{\zentrum(\mathbb{A})}}>> \Der_\mathbb{C}(\mathbb{C}[[t]])
\end{CD}
\]
\begin{cor}
The sequence
\begin{equation}\label{ses extended derivations}
0 \to \Der_{\mathbb{C}[[t]]}(\mathbb{A}) \to \mathfrak{L}(\mathbb{A}) \xrightarrow{(\cdot)\vert_{\zentrum(\mathbb{A})}} \mathfrak{l} \to 0
\end{equation}
is exact.
\end{cor}

The extension \eqref{ses extended derivations} spliced with the extension \eqref{ses derivations are inner} gives rise to the exact sequence
\begin{equation}\label{spliced extensions}
0 \to \dfrac1t\mathbb{C}[[t]] \to \dfrac1t\mathbb{A} \xrightarrow{\ad} \mathfrak{L}(\mathbb{A}) \xrightarrow{(\cdot)\vert_{\zentrum(\mathbb{A})}} \mathfrak{l} \to 0
\end{equation}
Since the extension \eqref{ses extended derivations} is split locally on $X$, the exact sequence \eqref{spliced extensions} gives rise to a $\shHom(\mathfrak{l},\dfrac1t\mathbb{C}[[t]])[1] \cong \dfrac1t\mathbb{C}[[t]][1]$-torsor (see \ref{subsection: Extensions and 2-torsors}) which we denote by $\shDES(\mathbb{A})$. 

A \emph{dilation equivariance structure (DES)} on (the symplectic DQ-algebra) $\mathbb{A}$ is a an object $\widetilde{\mathfrak{L}} \in \shDES(\mathbb{A})$. Explicitly, a DES on $\mathbb{A}$ is an extension $\widetilde{\mathfrak{L}}$ of $\mathfrak{l}$ by $\dfrac1t\mathbb{A}$  which lifts $\mathfrak{L}$, i.e. there is a commutative diagram of sheaves of vector spaces
\begin{equation}\label{dilation equivariance structure}
\begin{CD}
& & 0 & & 0 \\
& & @VVV @VVV \\
0 @>>> \dfrac1t\mathbb{C}[[t]] @= \dfrac1t\mathbb{C}[[t]] @>>> 0\\
& & @VVV @VVV @VVV\\
0 @>>> \dfrac1t\mathbb{A} @>>> \widetilde{\mathfrak{L}} @>>> \mathfrak{l} @>>> 0 \\
& & @V{\ad}VV @VVV @| \\
0 @>>> \Der_{\mathbb{C}[[t]]}(\mathbb{A}) @>>> \mathfrak{L}(\mathbb{A}) @>{(\cdot)\vert_{\zentrum(\mathbb{A})}}>> \mathfrak{l} @>>> 0 \\
& & @VVV @VVV @VVV \\
& & 0 & & 0 & & 0
\end{CD}
\end{equation}
with exact rows and columns.

\subsection{The canonical bracket on DES}
A DES $\widetilde{\mathfrak{L}}$ on $\mathbb{A}$ admits a canonical structure of a sheaf of Lie algebras. Let $D\in\widetilde{\mathfrak{L}}$ denote a locally defined section which projects to $t\partial_t$. Then, every section of $\widetilde{\mathfrak{L}}$ is of the form $a+\lambda\cdot D$, $a\in\dfrac1t\mathbb{A}$, $\lambda \in \mathbb{C}$. The bracket on $\widetilde{\mathfrak{L}}$ is defined by
\begin{equation}\label{canonical bracket}
[a_1+\lambda_1\cdot D,a_2+\lambda_2\cdot D] = [a_1,a_2] + \lambda_1\cdot\overline{D}(a_2) - \lambda_2\cdot\overline{D}(a_2) ,
\end{equation}
where $\overline{D}$ is the image of $D$ in $\mathfrak{L}(\mathbb{A})$.

\lemma
The bracket \eqref{canonical bracket} is independent of the choice of $D$.
\endlemma
\proof
Indeed, any other choice is of the form $D+b$ with $b\in\dfrac1t\mathbb{A}$. Then, $a_i+\lambda_i\cdot D = (a_i-\lambda_i b) + \lambda_i\cdot(D+b)$, $i = 1,2$, and
\begin{multline*}
[a_1-\lambda_1 b, a_2-\lambda_2 b] + \lambda_1\cdot\overline{(D+b)}(a_2-\lambda_1 b) - \lambda_2\cdot\overline{(D+b)}(a_2-\lambda_2 b) \\ = [a_1,a_2] + \lambda_1\cdot\overline{D}(a_2) - \lambda_2\cdot\overline{D}(a_2)  .
\end{multline*}
\endproof

Since the map $\widetilde{\mathfrak{L}} \to \mathfrak{l}$ admits a section locally on $X$, it follows that the Lie algebra structure described above is well-defined. We shall refer to the above Lie algebra structure as \emph{canonical}.

\begin{prop}
{\ }
\begin{enumerate}
\item The canonical bracket endows a DES with a structure of a Lie algebra.
\item All maps in the diagram \eqref{dilation equivariance structure} are morphisms of Lie algebras.
\item A morphism of DES is a morphism of Lie algebras with respect to the canonical brackets.
\end{enumerate}
\end{prop}
\proof
Direct calculation left to the reader.
\endproof

\subsection{The canonical action of units}\label{subsection: The canonical action of units}
Suppose that $\mathbb{A}$ is a symplectic DQ-algebra and $\widetilde{\mathfrak{L}}$ is a DES on $\mathbb{A}$. An automorphism $\phi \colon \mathbb{A} \to \mathbb{A}$ induces a Lie algebra automorphism $\Conj(\phi) \colon \Der_\mathbb{C}(\mathbb{A}) \to \Der_\mathbb{C}(\mathbb{A})$ defined by $D \mapsto \phi\circ D\circ\phi^{-1}$. The automorphism $\Conj(\phi)$ preserves the subalgebras $\Der_{\mathbb{C}[[t]]}(\mathbb{A})$ and $\mathfrak{L}(\mathbb{A})$ and induces the trivial automorphism on $\mathfrak{l}$.

The group $\mathbb{A}^\times$ of units (invertible elements of $\mathbb{A}$) acts on the latter by conjugation; for $u\in\mathbb{A}^\times$ we denote by $\Conj(u)$ the corresponding automorphism of $\mathbb{A}$. 

\lemma
Suppose that $\widetilde{\mathfrak{L}}$ is a DES on $\mathbb{A}$ and $u\in\mathbb{A}^\times$.
\begin{enumerate}
\item The automorphism of $\Der_\mathbb{C}(\mathbb{A})$ induced by $\Conj(u)$ is given by $D \mapsto D - \ad(D(u)\cdot u^{-1})$.

\item The formula
\begin{equation}\label{canonical action of units on DES}
D \mapsto D - [D,u]\cdot u^{-1} = D - \overline{D}(u)\cdot u^{-1} ,
\end{equation}
where $u\in\mathbb{A}^\times$ and $\overline{D}$ is the image of $D$ in $\mathfrak{L}(\mathbb{A})$, defines an action of $\mathbb{A}^\times$ on $\widetilde{\mathfrak{L}}$.

\item The maps $\dfrac1t\mathbb{A} \to \widetilde{\mathfrak{L}} \to \mathfrak{L}(\mathbb{A})$ are $\mathbb{A}^\times$-equivariant.
\end{enumerate}
\endlemma
\begin{proof}
Direct calculation left to the reader.
\end{proof}

We shall refer to the action defined by \eqref{canonical action of units on DES} as \emph{canonical}.

\subsection{Push-forward of DES}\label{subsection: DES pushforward algebras}
Suppose that $\phi \colon \mathbb{A}_1 \to \mathbb{A}_0$ is a morphism of symplectic DQ-algebras. The morphism $\phi$ restricts to the identity map between the respective centers identified with $\mathbb{C}[[t]]$ and induces the morphism of Lie algebras $\Conj(\phi) \colon \Der_\mathbb{C}(\mathbb{A}_1) \to \Der_\mathbb{C}(\mathbb{A}_0)$ defined by $D \mapsto \phi\circ D\circ\phi^{-1}$. Therefore, $\Conj(\phi)$ gives rise to the commutative diagram
\[
\begin{CD}
0 @>>> \dfrac1t\mathbb{C}[[t]] @>>> \dfrac1t\mathbb{A}_1 @>{\ad}>> \mathfrak{L}(\mathbb{A}_1) @>>> \mathfrak{l} @>>> 0 \\ 
& & @| @V{\phi}VV @V{\Conj(\phi)}VV @| \\
0 @>>> \dfrac1t\mathbb{C}[[t]] @>>> \dfrac1t\mathbb{A}_0 @>{\ad}>> \mathfrak{L}(\mathbb{A}_0) @>>> \mathfrak{l} @>>> 0
\end{CD}
\]
and, hence, a morphism of $\shHom(\mathfrak{l},\dfrac1t\mathbb{C}[[t]])[1] \cong \dfrac1t\mathbb{C}[[t]][1]$-torsors
\begin{equation}\label{DES pushforward}
\phi_* \colon \shDES(\mathbb{A}_1) \to \shDES(\mathbb{A}_0)
\end{equation}

\section{Symplectic DQ-algebroids} \label{section: Symplectic DQ-algebroids}

%\subsection{DQ-algebroids}
%Recall the definition of DQ-algebroids from \cite{KS}.

\begin{definition}[{\cite[Definition 2.3.1]{KS}}]
A \emph{DQ-algebroid} $\mathcal{C}$ is a $\mathbb{C}[[t]]$-algebroid such that for each open set $U\subseteq X$ with $\mathcal{C}(U) \neq \varnothing$ and any $L\in \mathcal{C}(U)$ the $\mathbb{C}[[t]]$-algebra $\shEnd_\mathcal{C}(L)$ is a DQ-algebra on $U$.
\end{definition}

In other words, a DQ-algebroid is a $\mathbb{C}[[t]]$-algebroid locally equivalent to a star-product.

%\subsection{The associated Poisson structure}\label{subsection: The associated Poisson structure}
\begin{prop}[{\cite[Proposition 7.3]{QCL}}]
There exists a unique Poisson bracket
\begin{equation*}
\{\cdot,\cdot\}^\mathcal{C} \colon \mathcal{O}_{X/\mathcal{P}}\otimes\mathcal{O}_{X/\mathcal{P}} \to \mathcal{O}_{X/\mathcal{P}}
\end{equation*}
such that for any $U\subset X$ with $\mathcal{C}(U)\neq\varnothing$ and any $L\in\mathcal{C}(U)$ the restriction of $\{\cdot,\cdot\}^\mathcal{C}$ to $U$ coincides with the Poisson bracket associated to the DQ-algebra $\shEnd_\mathcal{C}(L)$.
\end{prop}

\begin{notation}
We denote by $\pi^\mathcal{C} \in \Gamma(X;\bigwedge^2\mathcal{T}_{X/\mathcal{P}})$ the Poisson bi-vector which corresponds to $\{\cdot,\cdot\}^\mathcal{C}$.
\end{notation}

The assignment $\mathcal{C} \mapsto \pi^\mathcal{C}$ give rise to the canonical morphism
\begin{equation}\label{map: assoc poisson}
\DQ_{X/\mathcal{P}} \to \Lambda^2\mathcal{T}_{X/\mathcal{P}}
\end{equation}
For $\pi\in\Gamma(X;\Lambda^2\mathcal{T}_{X/\mathcal{P}})$ we denote by $\DQ^\pi_{X/\mathcal{P}}$ the fiber of \eqref{map: assoc poisson} over $\pi$.

\subsection{Symplectic DQ-algebroids}
If $\omega\in\Gamma(X;\Omega^{2,cl}_{X/\mathcal{P}})$ is a symplectic form and $\pi = \omega^{-1}$ we shall denote by $\DQ^\omega_{X/\mathcal{P}}$ the 2-stack $\DQ^\pi_{X/\mathcal{P}}$.

\begin{example}\label{ex:MW defines a sympl DQ-algebroid}
	Let $(X, \omega)$ be a symplectic manifold. Suppose $U \subset X$ is an open subset which admits coordinates satisfying canonical relations as in Example \ref{example: MW}. We denote by $\mathbb{A}_{MW} := (\mathcal{O}_{U/\mathcal{P}}[[t]], \star)$ the corresponding Moyal-Weyl algebra. Then $\mathbb{A}_{MW}^+$ (see Appendix \ref{subsection: algebroids}) is an object of $\DQ^\omega_{X/\mathcal{P}}(U)$.
\end{example}
 
\begin{prop}\label{prop: symp dq 2-gerbe}
{~}
\begin{enumerate}
\item For any symplectic DQ-algebroid $\mathcal{C}$ the canonical morphism $\mathbb{C}[[t]] \to \zentrum(\mathcal{C})$ is an isomorphism.

\item For any symplectic DQ-algebroid $\mathcal{C}$ the canonical morphism (see Appendix \ref{subsection: algebroids}) $\left.\mathbb{C}[[t]]\right.^\times[1] \to \shAut_{\DQ^\omega_{X/\mathcal{P}}}(\mathcal{C})$ and the morphism $\shAut_{\DQ^\omega_{X/\mathcal{P}}}(\mathcal{C}) \to \left.\mathbb{C}[[t]]\right.^\times[1] \colon F \mapsto \shHom_{\shAut(\mathcal{C})}(\id, F)$ are mutually quasi-inverse monoidal equivalences.

\item The 2-stack $\DQ^\omega_{X/\mathcal{P}}$ is a $\mathbb{C}[[t]]^\times[1]$-gerbe (see Appendix \ref{subsection: 2-torsors}).
\end{enumerate}
\end{prop}
\begin{proof}
Recall that a symplectic DQ-algebroid $\mathcal{C}$ is locally equivalent to $\mathbb{A}^+$, where $\mathbb{A}$ is a symplectic DQ-algebra.
\begin{enumerate}
\item The problem is local on $X$. Therefore, $\zentrum(\mathcal{C}) \cong \zentrum(\mathbb{A})$ and the claim follows from (2) of  Corollary \ref{cor: props of DQ alg}.

\item The claim follows from the fact that all automorphisms of $\mathbb{A}$ are inner by (4) of Corollary \ref{cor: props of DQ alg}. 

\item In view of the equivalence $\shAut(\mathcal{C}) \to \left.\mathbb{C}[[t]]\right.^\times[1]$ it remains to show that $\DQ^\omega_{X/\mathcal{P}}$ is locally non-empty and locally-connected. Example \ref{ex:MW defines a sympl DQ-algebroid} shows that $\DQ^\omega_{X/\mathcal{P}}$ is locally non-empty. Local connectedness follows from (1) of Corollary \ref{cor: props of DQ alg}.

\end{enumerate}
\end{proof}

%For a manifold $X$ let $\bigwedge^2\mathcal{T}_X^\symp$ denote the sheaf of symplectic, which is to say, non-degenerate Poisson bi-vectors; let $\DQ_X^\symp$ denote the substack of symplectic DQ-algebroids so that ... restricts to  $\DQ_X^\symp \to \bigwedge^2\mathcal{T}_X^\symp$. The assignment $\bigwedge^2\mathcal{T}_X^\symp \ni \pi \mapsto \int t\pi$ defines a map $\bigwedge^2\mathcal{T}_X^\symp \to \dfrac1t k[2]$. The composition of the two give the morphism
%\begin{equation}\label{assoc gerbe symp}
%\DQ_X^\symp \to \dfrac1t k[2] \ .
%\end{equation}

\subsection{DES enhancement}\label{subsection: DES enhancement}
Suppose that $\mathcal{C}$ is a symplectic DQ-algebroid. 

For $L_i \in \mathcal{C}$, $\widetilde{\mathfrak{L}}_i \in \DES(\shEnd_\mathcal{C}(L_i))$, $i=0,1$, a morphism $(L_1,\widetilde{\mathfrak{L}}_1) \to (L_0,\widetilde{\mathfrak{L}}_0)$ is a pair $(f,s)$, where $f \colon L_1 \to L_0$ is an isomorphism in $\mathcal{C}$ and $s \colon \widetilde{\mathfrak{L}}_1 \to \widetilde{\mathfrak{L}}_0$ is a morphism of sheaves of $k$-vector spaces such that the diagrams
\[
\begin{CD}
\dfrac1t\shEnd_\mathcal{C}(L_1) @>>> \widetilde{\mathfrak{L}}_1 @>>> \mathfrak{l} \\
@V{\Conj(f)}VV @V{s}VV @| \\
\dfrac1t\shEnd_\mathcal{C}(L_0) @>>> \widetilde{\mathfrak{L}}_0 @>>> \mathfrak{l}
\end{CD}
\]
and
\[
\begin{CD}
\dfrac1t\mathbb{C}[[t]] @>>>  \widetilde{\mathfrak{L}}_1 @>>> \mathfrak{L}(\shEnd_\mathcal{C}(L_1)) \\
@| @V{s}VV @V{\Conj(f)}VV \\
\dfrac1t\mathbb{C}[[t]] @>>>  \widetilde{\mathfrak{L}}_0 @>>> \mathfrak{L}(\shEnd_\mathcal{C}(L_0))
\end{CD}
\]
are commutative. Composition of morphisms is defined in the obvious way. 

We denote by $\widetilde{\mathcal{C}}$ the stack with (locally defined) objects pairs $(L,\widetilde{\mathfrak{L}})$, $L \in \mathcal{C}$, $\widetilde{\mathfrak{L}} \in \DES(\shEnd_\mathcal{C}(L))$ and morphisms defined as above. The assignment $(L,\widetilde{\mathfrak{L}}) \mapsto L$ extends to morphism
\[
p_\mathcal{C} \colon \widetilde{\mathcal{C}} \to i\mathcal{C}
\]
which makes $\widetilde{\mathcal{C}}$ a category cofibered in $\shgpdExt^1(\mathfrak{l},\dfrac1t\mathbb{C}[[t]])$-torsors over $i\mathcal{C}$.

\subsection{Functoriality of the DES enhancement}\label{subsection: Functoriality of the DES enhancement}
Suppose that $F \colon \mathcal{C}_1 \to \mathcal{C}_0$ is a 1-morphism of symplectic DQ-algebroids. For $L \in \mathcal{C}_1$, the morphism of DQ-algebras
\[
F \colon \shEnd_{\mathcal{C}_1}(L) \to \shEnd_{\mathcal{C}_0}(F(L))
\]
gives rise to the morphism of stacks
\[
F_* \colon \shDES(\shEnd_{\mathcal{C}_1}(L)) \to \shDES(\shEnd_{\mathcal{C}_0}(F(L))) \ .
\]
as in \ref{subsection: DES pushforward algebras}.

The assignment $(L,\widetilde{\mathfrak{L}}) \mapsto \widetilde{F}(L,\widetilde{\mathfrak{L}}) := (F(L), F_*\widetilde{\mathfrak{L}})$ extends to a morphism of stacks $\widetilde{F} \colon \widetilde{\mathcal{C}_1} \to \widetilde{\mathcal{C}_0}$ such that the diagram
\[
\begin{CD}
\widetilde{\mathcal{C}_1} @>{\widetilde{F}}>> \widetilde{\mathcal{C}_0} \\
@V{p_{\mathcal{C}_1}}VV @VV{p_{\mathcal{C}_0}}V \\
i\mathcal{C}_1 @>{F}>> i\mathcal{C}_0
\end{CD}
\]
is commutative and represents a morphism of $\shgpdExt^1(\mathfrak{l},\mathbb{C}[[t]])$-torsors.

\subsection{Action of 2-morphisms}
Suppose that $F_i \colon \mathcal{C}_1 \to \mathcal{C}_0$, $i=0,1$, are 1-morphisms  of symplectic DQ-algebroids and $f \colon F_1 \to F_0$ is a 2-morphism. For $L\in\mathcal{C}_1$ the composition
\[
\shEnd_{\mathcal{C}_1}(L) \xrightarrow{F_1} \shEnd_{\mathcal{C}_0}(F_1(L)) \xrightarrow{\Conj(f)} \shEnd_{\mathcal{C}_0}(F_0(L))
\]
coincides with the map $\shEnd_{\mathcal{C}_1}(L) \xrightarrow{F_0} \shEnd_{\mathcal{C}_0}(F_0(L))$. Therefore, for $(L,\widetilde{\mathfrak{L}}) \in \widetilde{\mathcal{C}_1}$ there is a canonical isomorphism $\Conj(f)_*F_{1*}\widetilde{\mathfrak{L}} \cong F_{0*}\widetilde{\mathfrak{L}}$. Let $f_{*L} \colon F_{1*}\widetilde{\mathfrak{L}} \to F_{0*}\widetilde{\mathfrak{L}}$ denote the composition $F_{1*}\widetilde{\mathfrak{L}} \to \Conj(f)_*F_{1*}\widetilde{\mathfrak{L}} \xrightarrow{\cong} F_{0*}\widetilde{\mathfrak{L}}$. The 2-morphism $f_* \colon F_{1*} \to F_{0*}$ is defined by $L \mapsto f_{*L}$.

The 2-morphism $\widetilde{F_1} \to \widetilde{F_0}$ induced by $f \colon F_1 \to F_0$ is defined by
\[
\widetilde{f} := (f,f_*) \colon \widetilde{F_1} \to \widetilde{F_0} .
\]

%\subsection{Dilation Equivariance Structures}
\begin{definition}
A \emph{dilation equivariance structure (DES) on a symplectic DQ-algebroid $\mathcal{C}$} is a section $\nabla \colon i\mathcal{C} \to \widetilde{\mathcal{C}}$ of the projection $p_\mathcal{C} \colon \widetilde{\mathcal{C}} \to i\mathcal{C}$ such that for $L\in\mathcal{C}$ the induced map
\[
\shAut_\mathcal{C}(L) = \shEnd_\mathcal{C}\mathcal(L)^\times \to \shAut_{\widetilde{\mathcal{C}}}(\nabla(L))
\]
coincides with the canonical action \ref{subsection: The canonical action of units}.
\end{definition}

A morphism $f \colon \nabla_1 \to \nabla_0$ of DES on $\mathcal{C}$ is a morphism of sections of the projection $p_\mathcal{C}$, i.e. $f\diamond\id_{p_\mathcal{C}} = \id_{\id_\mathcal{C}}$.

%such that the diagram
%\[
%\begin{CD}
%\shAut_\mathcal{C}(L) @>>> \shAut_{\widetilde{\mathcal{C}}}(L,\nabla_1(L)) \\
%@| @VV{\Conj(f)}V \\
%\shAut_\mathcal{C}(L) @>>> \shAut_{\widetilde{\mathcal{C}}}(L,\nabla_0(L))
%\end{CD}
%\]
%is commutative. 
We denote the category of DES on $\mathcal{C}$ by $\DES(\mathcal{C})$. The assignment $U \mapsto \DES(\mathcal{C}\vert_U)$ extends to a stack in groupoids which we denote $\shDES(\mathcal{C})$.

%\lemma
%For $\nabla \in \DES(\mathcal{C})$ the natural map $\dfrac1t\mathbb{C}[[t]] \to \Aut_{\shDES(\mathcal{C})}(\nabla)$ is an isomorphism.
%\endlemma
%\begin{proof}
%.................................
%\end{proof}

\subsection{Symplectic DQ-algebroids with DES}\label{subsection: Symplectic DQ-algebroids with DES}
Symplectic DQ-algebroids equipped with DES form a 2-category in the following manner.
\begin{itemize}
\item[\emph{Objects}] The objects are pairs $(\mathcal{C}, \nabla)$, where $\mathcal{C}$ is a symplectic DQ-algebroid and $\nabla \in \DES(\mathcal{C})$.

\item[\emph{1-morphisms:}]
Suppose that $\mathcal{C}_i$ are symplectic DQ-algebroids  and $\nabla_i \in \DES(\mathcal{C}_i)$, $i=0,1$, a 1-morphism $(\mathcal{C}_1, \nabla_1) \to (\mathcal{C}_0, \nabla_0)$ is a pair $(F,\theta)$, where $F \colon \mathcal{C}_1 \to \mathcal{C}_0$ is a 1-morphism and
\[
\theta \colon \widetilde{F}\circ\nabla_1 \to \nabla_0\circ F
\]
is a 2-morphism between 1-morphisms $\mathcal{C}_1 \to \widetilde{\mathcal{C}_0}$ which is compatible with the respective projections, that is, the composition
\[
F = p_{\mathcal{C}_0}\circ\widetilde{F}\circ\nabla_1 \xrightarrow{\id_{p_{\mathcal{C}_0}}\diamond\,\theta} p_{\mathcal{C}_0}\circ\nabla_0\circ F
\]
coincides with $\id_F$.

\item[\emph{Composition:}]
Suppose that $(\mathcal{C}_2, \nabla_2) \xrightarrow{(G,\rho)} (\mathcal{C}_1, \nabla_1) \xrightarrow{(F,\theta)} (\mathcal{C}_0, \nabla_0)$ are 1-morphisms. The composition
\begin{equation}\label{composition 1-morphism des}
\nabla_2\circ\widetilde{G\circ F} = \nabla_2\circ\widetilde{G}\circ\widetilde{F} \xrightarrow{\rho\diamond\id_{\widetilde{F}}} G\circ\nabla_1\circ\widetilde{F} \xrightarrow{\id_G\diamond\theta} G\circ F\circ\nabla_0
\end{equation}
gives rise to the 1-morphism $(\mathcal{C}_2, \nabla_2) \to (\mathcal{C}_0, \nabla_0)$ defined to be the composition $(G,\rho)\circ(F,\theta)$.

\item[\emph{Identity:}]
The identity morphism $\id_{(\mathcal{C}, \nabla)}$ is given by the pair $(\id_{\mathcal{C}}, \id_{\nabla})$.

\item[\emph{2-morphisms:}]
Suppose that $(F_i,\theta_i)$, $i=0,1$, are 1-morphisms $(\mathcal{C}_1, \nabla_1) \to (\mathcal{C}_0, \nabla_0)$. A 2-morphism $f \colon (F_1,\theta_1) \to (F_0,\theta_0)$ is a 2-morphism $f \colon F_1 \to F_0$ such that the diagram
\[
\begin{CD}
\widetilde{F_1}\circ\nabla_1 @>{\theta_1}>> \nabla_0\circ F_1 \\
@V{\widetilde{f}\diamond\id_{\nabla_1}}VV @VV{\id_{\nabla_0}\diamond f}V \\
\widetilde{F_0}\circ\nabla_1 @>{\theta_0}>> \nabla_0\circ F_0
\end{CD}
\]
is commutative.
\item[\emph{Composition:}]
Horizontal and vertical composition of 2-morphism are defined as those in the 2-category of categories.
\end{itemize}

For a symplectic form $\omega$ we denote by $\widetilde{\DQ_{X/\mathcal{P}}^\omega}$ the 2-stack with (locally defined) objects pairs $(\mathcal{C}, \nabla)$ with $\mathcal{C}\in\DQ_{X/\mathcal{P}}^\omega$, $\nabla\in\DES(\mathcal{C})$,   and 1- and 2-morphism as above. The assignment $(\mathcal{C}, \nabla) \mapsto \mathcal{C}$ extends to a morphism $\widetilde{\DQ_{X/\mathcal{P}}^\omega} \to \DQ_{X/\mathcal{P}}^\omega$.

\begin{example}\label{example: DES on MW}
Suppose that $\mathbb{A}$ is a symplectic DQ-algebra. Let $\widetilde{\mathfrak{L}} \in \DES(\mathbb{A})$. For an open subset $U \subseteq X$, the assignment $\ast \mapsto (\ast, \widetilde{\mathfrak{L}}\vert_U)$ extends to a functor $\nabla_{\widetilde{\mathfrak{L}},U}\colon \mathbb{A}(U)^+ \to \widetilde{\mathbb{A}^+}(U)$. Since $\widetilde{\mathbb{A}^+}$ is a stack, as $U$ varies the functors $\nabla_{\widetilde{\mathfrak{L}},U}$ give rise to the morphism of stacks $\nabla_{\widetilde{\mathfrak{L}}} \colon \mathbb{A}^+ \to \widetilde{\mathbb{A}^+}$ which is a DES. The assignment $\widetilde{\mathfrak{L}} \mapsto \nabla_{\widetilde{\mathfrak{L}}}$ defines a morphism of stacks $\shDES(\mathbb{A}) \to \shDES(\mathbb{A}^+)$ which is an equivalence. A quasi-inverse is given by $\nabla \mapsto \nabla(\ast)$.
\end{example}

\subsection{}
Suppose that $\nabla$ is a DES on a symplectic DQ-algebroid $\mathcal{C}$ and $(F,\theta) \in \Aut(\mathcal{C},\nabla)$. Horizontal composition with $\id_\nabla$ gives rise to the map $\mathbb{C}[[t]]$-modules
\[
\shHom_{\shAut(\mathcal{C})}(\id_\mathcal{C},F) \to \shHom_{\shHom(i\mathcal{C},\widetilde{C})}(\nabla,\nabla\circ F) \colon \psi\mapsto\id_\nabla\diamond\psi .
\]
Another map is given by the composition
\begin{multline*}
\shHom_{\shAut(\mathcal{C})}(\id_\mathcal{C},F) \xrightarrow{\widetilde{(\cdot)}} \shHom_{\shAut(\widetilde{\mathcal{C}})}(\id_{\widetilde{\mathcal{C}}}, \widetilde{F}) \xrightarrow{(\cdot)\diamond\id_\nabla} \\
\shHom_{\shHom(i\mathcal{C},\widetilde{C})}(\nabla, \widetilde{F}\circ\nabla) \xrightarrow{\theta\circ(\cdot)} \shHom_{\shHom(i\mathcal{C},\widetilde{C})}(\nabla,\nabla\circ F)
\end{multline*}
Since, for $\psi\in\shHom_{\shAut(\mathcal{C})}(\id_\mathcal{C},F)$,
\begin{multline*}
\id_{p_{\mathcal{C}}}\diamond(\theta\circ(\widetilde{\psi}\diamond\id_\nabla)) = (\id_{p_{\mathcal{C}}}\diamond\theta)\circ(\id_{p_{\mathcal{C}}}\diamond\id_\nabla\diamond\psi) = \\
\id_F\circ(\psi\diamond\id_{p_{\mathcal{C}}}\diamond\id_\nabla)=\psi = \id_{p_{\mathcal{C}}}\diamond(\id_\nabla\diamond\psi)
\end{multline*}
it follows that $\theta\circ(\widetilde{\psi}\diamond\id_\nabla)$ and $\id_\nabla\diamond\psi$ are in the same orbit of the action of $\shAut_{\DES(\mathcal{C})}(\nabla) = \dfrac1t\mathbb{C}[[t]]$ on $\shHom_{\shHom(i\mathcal{C},\widetilde{C})}(\nabla, \nabla\circ F)$. Therefore, the difference map
\begin{equation}\label{curvature map for an auto}
c(F,\theta) \colon \shHom_{\shAut(\mathcal{C})}(\id_\mathcal{C}, F) \to \dfrac1t\mathbb{C}[[t]]
\end{equation}
given by the formula
\[
c(F,\theta)(\psi) = \theta\circ(\widetilde{\psi}\diamond\id_\nabla) - \id_\nabla\diamond\psi
\]
is defined.

\lemma\label{lemma: nabla is curvature}
The map \eqref{curvature map for an auto} satisfies
\[
c(F,\theta)(f\cdot\psi) = c(F,\theta)(\psi) + t\partial_t\log(f) \ ,
\]
where $\psi\in\shHom_{\shAut(\mathcal{C})}(\id_\mathcal{C}, F)$ and $f\in\mathbb{C}[[t]]^\times$.
\endlemma
\begin{proof}
Since $\Conj(f\cdot\psi) = \Conj(f)$, it follows that that for any $L\in\mathcal{C}$ the maps $\nabla(L) \to \nabla(F(L))$ induced by $\theta\circ(\widetilde{\psi}\diamond\id_\nabla)$ and $\theta\circ(\widetilde{f\cdot\psi}\diamond\id_\nabla)$ coincide. Therefore,
\begin{multline*}
c(F,\theta)(f\cdot\psi) - c(F,\theta)(\psi) = \\
(\theta\circ(\widetilde{f\cdot\psi}\diamond\id_\nabla) - \id_\nabla\diamond(f\cdot\psi)) - (\theta\circ(\widetilde{f\psi}\diamond\id_\nabla) - \id_\nabla\diamond\psi) = \\
-(\id_\nabla\diamond(f\cdot\psi) - \id_\nabla\diamond\psi) = -\id_\nabla\diamond(f\cdot\psi - \psi)
\end{multline*}
The latter expression is calculated using the embedding $\mathbb{C}[[t]]^\times \subset \Aut(L)$, and the canonical action of units \eqref{canonical action of units on DES}. Namely, for $f\in\mathbb{C}[[t]]^\times$ and $D\in\nabla(L)$ such that $D\vert_{\zentrum(\End_\mathbb{C}(L))} = t\dfrac{d\ }{dt} = t\partial_t$, the canonical action is given by
\[
D \mapsto D - \overline{D}(f)\cdot f^{-1} = D - t\partial_t\log(f)
\]
which implies the desired result.
\end{proof}
\begin{cor}
The pair $(\shHom_{\shAut(\mathcal{C})}(\id_\mathcal{C}, F), c(F,\theta))$ is a $(\mathbb{C}[[t]]^\times \xrightarrow{t\partial_t\log}\dfrac1t\mathbb{C}[[t]])$-torsor.
\end{cor}

The assignment $(F,\theta) \mapsto (\shHom_{\shAut(\mathcal{C})}(\id_\mathcal{C}, F), c(F,\theta))$ extends to a morphism
\begin{equation}\label{automorphisms with DES to torsors with curv}
\shAut_{\widetilde{\DQ_{X/\mathcal{P}}^\omega}}(\mathcal{C},\nabla) \to (\mathbb{C}[[t]]^\times \xrightarrow{t\partial_t\log}\dfrac1t\mathbb{C}[[t]])[1] \ .
\end{equation}
The canonical morphism $\mathbb{C}[[t]]^\times[1] \to \shAut_{\DQ_{X/\mathcal{P}}^{\omega}}(\mathcal{C})$ lifts to the morphism
\begin{equation}\label{torsors with curv to automorphisms with DES}
(\mathbb{C}[[t]]^\times \xrightarrow{t\partial_t\log}\dfrac1t\mathbb{C}[[t]])[1] \to \shAut_{\widetilde{\DQ_{X/\mathcal{P}}^\omega}}(\mathcal{C},\nabla) \colon (T,c) \mapsto (T\otimes(\cdot), c).
\end{equation}
\begin{prop}\label{prop: algds with des 2-torsor}
{~}
\begin{enumerate}
\item The morphism \eqref{torsors with curv to automorphisms with DES} and the morphism \eqref{automorphisms with DES to torsors with curv} are mutually quasi-inverse monoidal equivalences.

\item The 2-stack $\widetilde{\DQ_{X/\mathcal{P}}^\omega}$ is a $(\mathbb{C}[[t]]^\times \xrightarrow{t\partial_t\log}\dfrac1t\mathbb{C}[[t]])[1]$-gerbe.
\end{enumerate}
\end{prop}
\begin{proof}
It is clear that \eqref{torsors with curv to automorphisms with DES} and the morphism \eqref{automorphisms with DES to torsors with curv} are mutually quasi-inverse.

In view of the equivalence $\shAut_{\widetilde{\DQ_{X/\mathcal{P}}^\omega}}(\mathcal{C},\nabla) \cong (\mathbb{C}[[t]]^\times \xrightarrow{t\partial_t\log}\dfrac1t\mathbb{C}[[t]])[1]$ it remains to show that $\widetilde{\DQ_{X/\mathcal{P}}^\omega}$ is locally non-empty and locally connected.
These properties follow from the fact that any symplectic DQ-algebroid is locally equivalent to one of the form $\mathbb{A}^+$, where $\mathbb{A}$ is a symplectic DQ-algebra, and 
Example \ref{example: DES on MW}. 
\end{proof}

\section{Classical limits of DES} \label{section: Classical limits of DES}
\subsection{Classical limit of DES on DQ-algebras}\label{subsection: Classical limit of DES}
Suppose that $\mathbb{A}$ is a symplectic DQ-algebra. Let $\pi$ denote the associated Poisson bi-vector and let $\omega = \pi^{-1}$ denote the corresponding symplectic form.

Let $F_i\mathbb{A} = t^{-i}\mathbb{A}$. The $t$-adic filtration $F_\bullet\mathbb{A}$ induces the filtration denoted $F_\bullet\Der_\mathbb{C}(\mathbb{A})$, hence filtrations $F_\bullet\Der_{\mathbb{C}[[t]]}(\mathbb{A})$, $F_\bullet\mathfrak{L}(\mathbb{A})$.
\lemma\label{lemma: conf vf and Gr}
{\ }
\begin{enumerate}
\item $\Gr_\bullet^F\mathbb{A} = \mathcal{O}_{X/\mathcal{P}}[t]$ with the Poisson structure given by $t\pi$.

\item
\[
\Der^i_\mathbb{C}(\mathcal{O}_{X/\mathcal{P}}[t]) = 
\begin{cases}
0 & \mbox{if } i\geq 2 \\
\mathcal{O}_X\cdot\partial_t & \mbox{if } i = 1 \\
t^{-i}\left(\mathcal{T}_{X/\mathcal{P}}\oplus\mathcal{O}_{X/\mathcal{P}}\otimes\mathfrak{l}\right) & \mbox{if } i\leq 0
\end{cases}
\]
\end{enumerate}
\endlemma
\begin{proof}
Left to the reader.
\end{proof}

Let $\Der^0_\mathbb{C}(\mathcal{O}_{X/\mathcal{P}}[t], t\pi)$ denote the subsheaf of $\Der^0_\mathbb{C}(\mathcal{O}_{X/\mathcal{P}}[t])$ of derivations which preserve $t\pi$, i.e. those $D \in \Der^0_\mathbb{C}(\mathcal{O}_{X/\mathcal{P}}[t])$ which satisfy $[D,t\pi]_{SN} = 0$. Here and below we denote by $[\ ,\ ]_{SN}$ the Schouten-Nijenhuis bracket.
For $\lambda\in \mathbb{C}$,
\[
[t\pi, \xi + \lambda\cdot t\partial_t]_{SN} = t([\pi,\xi]_{SN} + \lambda\pi)\ ,
\]
i.e. $\xi + \lambda\cdot t\partial_t\in\Der^0_\mathbb{C}(\mathcal{O}_{X/\mathcal{P}}[t], t\pi) \Leftrightarrow [\pi,\xi]_{SN} + \lambda\pi=0$.

The sheaf $\mathcal{T}_{X/\mathcal{P}}^{\pi-\mathtt{conf}}$ of ($\pi$-)conformal vector fields is defined by the pull-back square
\[
\begin{CD}
\mathcal{T}_{X/\mathcal{P}}^{\pi-\mathtt{conf}} @>>> \Der^0_\mathbb{C}(\mathcal{O}_{X/\mathcal{P}}[t], t\pi) \\
@VVV @VVV \\
\mathfrak{l} @>{1\otimes\id}>> \mathcal{O}_{X/\mathcal{P}}\otimes\mathfrak{l}
\end{CD}
\]
The map $\mathcal{T}_{X/\mathcal{P}}^{\pi-\mathtt{conf}} \to \mathcal{T}_{X/\mathcal{P}}$ is a monomorphism so that we can and will regard the former as a subsheaf of the latter.
\lemma\label{lemma: Gr Der A and Der O}
{\ }
\begin{enumerate}
\item The map $\Gr^F_0\Der_\mathbb{C}(\mathbb{A}) \to \Der^0_\mathbb{C}(\mathcal{O}_{X/\mathcal{P}}[t])$ is injective with image $\Der^0_\mathbb{C}(\mathcal{O}_{X/\mathcal{P}}[t], t\pi)$.

\item The map $\Gr^F_0\Der_{\mathbb{C}[[t]]}(\mathbb{A}) \to \Gr^F_0\Der_\mathbb{C}(\mathbb{A})$ is injective and the image of the composition $\Gr^F_0\Der_{\mathbb{C}[[t]]}(\mathbb{A}) \to \Gr^F_0\Der_\mathbb{C}(\mathbb{A}) \to \Der^0_\mathbb{C}(\mathcal{O}_{X/\mathcal{P}}[t])$ is equal to $\mathcal{T}_{X/\mathcal{P}}^\pi$.

\item The map $\Gr^F_0\mathfrak{L}(\mathbb{A}) \to \Der^0_\mathbb{C}(\mathcal{O}_{X/\mathcal{P}}[t])$ is injective with image $\mathcal{T}_{X/\mathcal{P}}^{\pi-\mathtt{conf}}$.
\end{enumerate}
\endlemma
\begin{proof}
Left to the reader.
\end{proof}

Since $\pi$ is nondegenerate, the Lichnerowicz-Poisson complex
\begin{equation}\label{Lichnerowicz-Poisson complex}
0 \to \dfrac1t\mathcal{O}_{X/\mathcal{P}} \xrightarrow{[t\pi,\,\cdot\,]_{SN}}
\mathcal{T}_{X/\mathcal{P}} \xrightarrow{[t\pi,\,\cdot\,]_{SN}} t\cdot\textstyle{\bigwedge^2}\mathcal{T}_{X/\mathcal{P}} \xrightarrow{[t\pi,\,\cdot\,]_{SN}}\cdots
\end{equation}
satisfies the Poincar\'e Lemma.

In what follows we shall identify $\mathfrak{l}$ with $\mathbb{C}$ using the global section $t\partial_t$. Hence, there is a canonical map $\mathcal{T}_{X/\mathcal{P}}^{\pi-\mathtt{conf}} \to \mathbb{C}$.
\begin{prop}
The diagram
\begin{equation}\label{poisson-de Rham}
\begin{CD}
0 @>>> \dfrac1t\mathbb{C} @>>> \dfrac1t\mathcal{O}_X @>{[t\pi,\,\cdot\,]_{SN}}>> \mathcal{T}_{X/\mathcal{P}}^{\pi-\mathtt{conf}} @>>> \mathbb{C} @>>> 0 \\
& & @| @| @VVV @VV{1\mapsto t\pi}V \\
0 @>>> \dfrac1t\mathbb{C} @>>> \dfrac1t\mathcal{O}_{X/\mathcal{P}} @>{[t\pi,\,\cdot\,]_{SN}}>> 
\mathcal{T}_{X/\mathcal{P}} @>{[t\pi,\,\cdot\,]_{SN}}>> t\cdot\textstyle{\bigwedge^2}\mathcal{T}_{X/\mathcal{P}} @>{[t\pi,\,\cdot\,]_{SN}}>>\cdots
\end{CD}
\end{equation}
where the bottom row is the complex \eqref{Lichnerowicz-Poisson complex}, is a pull-back diagram with exact rows. Thus, the top row represents the class of $\dfrac1t\omega$ in $H^2(X;\dfrac1t\mathbb{C})$.
\end{prop}
\begin{proof}
Follows from Lemma \ref{lemma: conf vf and Gr}, Lemma \ref{lemma: Gr Der A and Der O} and the Poincar\'e Lemma.
\end{proof}

The principal symbol map $\sigma_\mathbb{A} \colon \mathbb{A} \to \mathcal{O}_X$ gives rise to the map of exact sequences
\[
\begin{CD}
0 @>>> \dfrac1t\mathbb{C}[[t]] @>>> \dfrac1t\mathbb{A}  @>{\ad}>> \mathfrak{L}(\mathbb{A}) @>{(\cdot)\vert_{\zentrum(\mathbb{A})}}>> \mathfrak{l} @>>> 0 \\
& & @V{\sigma}VV @V{\sigma_\mathbb{A}}VV @V{\sigma_\mathbb{A}}VV @VV{t\partial_t \mapsto 1}V \\
0 @>>> \dfrac1t\mathbb{C} @>>> \dfrac1t\mathcal{O}_{X/\mathcal{P}} @>{[t\pi,\,\cdot\,]_{SN}}>> \mathcal{T}_{X/\mathcal{P}}^{\pi-\mathtt{conf}} @>>> \mathbb{C} @>>> 0
\end{CD}
\]
hence to the morphism of torsors (Appendix \ref{subsection: Extensions and 2-torsors})
\begin{equation}\label{DES to int}
\gr \colon \shDES(\mathbb{A}) \to \int\frac1t\omega ,
\end{equation}
where $\int\frac1t\omega$ denotes the $\dfrac1t\mathbb{C}[1]$-torsor corresponding to the top row of \eqref{poisson-de Rham}, relative to the morphism of Picard stacks $ \dfrac1t\mathbb{C}[[t]][1] \xrightarrow{\sigma} \dfrac1t\mathbb{C}[1]$.

The morphism \eqref{DES to int} is natural in $\mathbb{A}$ in the sense that it commutes with the direct image functors \eqref{DES pushforward}.

\subsection{Des obstructions parasites}
Since the map of complexes
\[
(\mathbb{C}^\times \xrightarrow{0} \dfrac1t\mathbb{C} \oplus \mathbb{C}) \to (\mathbb{C}[[t]]^\times \xrightarrow{t\partial_t\log}\dfrac1t\mathbb{C}[[t]])
\]
is quasi-isomorphism, it follows that there is a canonical equivalence
\begin{equation}\label{parasites}
(\mathbb{C}[[t]]^\times \xrightarrow{t\partial_t\log}\dfrac1t\mathbb{C}[[t]])[3] \cong \mathbb{C}^\times[3]\times \dfrac1t\mathbb{C}[2]\times\mathbb{C}[2] \ .
\end{equation}
Thus, $\widetilde{\DQ_{X/\mathcal{P}}^\omega}$ determines a class in $\pi_0(\mathbb{C}^\times[3](X)\times \dfrac1t\mathbb{C}[2](X)\times\mathbb{C}[2](X)) = H^3(X;\mathbb{C}^\times)\times H^2(X;\dfrac1t\mathbb{C})\times H^2(X;\mathbb{C})$. In \cite{D}, 4.7, P.~Deligne refers to the last two components as \emph{obstructions parasites}.

\subsection{Classical limit of DES on DQ-algebroids}
For $\mathcal{C}\in\DQ_{X/\mathcal{P}}^\omega$ the assignment $\widetilde{\mathcal{C}} \ni (L,\widetilde{\mathfrak{L}}) \mapsto \gr\widetilde{\mathfrak{L}}$ (cf. \ref{subsection: DES enhancement} and \eqref{DES to int}) extends to a functor
\[
\widetilde{\gr}_\mathcal{C} \colon \widetilde{\mathcal{C}} \to \int\frac1t\omega
\]
which is equivariant with respect to the morphism of Picard stacks
\[
\shgpdExt^1(\mathfrak{l},\dfrac1t\mathbb{C}[[t]]) \cong \dfrac1t\mathbb{C}[[t]][1] \to \dfrac1t\mathbb{C}[1]
\]
induced by reduction modulo $t$.

For a 1-morphism $F \colon \mathcal{C}_1 \to \mathcal{C}_0$ and $(L,\widetilde{\mathfrak{L}})\in\widetilde{\mathcal{C}_1}$ the map $\shEnd_{\mathcal{C}_1}(L) \to \shEnd_{\mathcal{C}_0}(F(L))$ induces the identity map after reduction modulo $t$. Therefore, the isomorphism $\widetilde{\mathfrak{L}} \to F_*\widetilde{\mathfrak{L}}$ of $\mathbb{C}[[t]]$-modules induces the identity map $\gr\widetilde{\mathfrak{L}} \to \gr F_*\widetilde{\mathfrak{L}}$. Moreover, the functor $\widetilde{F} \colon \widetilde{\mathcal{C}_1} \to \widetilde{\mathcal{C}_0}$ satisfies $\widetilde{\gr}_{\mathcal{C}_0}\circ\widetilde{F} = \widetilde{\gr}_{\mathcal{C}_1}$.

For $(\mathcal{C},\nabla) \in \widetilde{\DQ_{X/\mathcal{P}}^\omega}$ the composition $i\mathcal{C} \xrightarrow{\nabla} \widetilde{\mathcal{C}} \xrightarrow{\widetilde{\gr}_\mathcal{C}} \int\frac1t\omega$ is locally constant. For a 1-morphism $(F,\theta) \colon (\mathcal{C}_1,\nabla_1) \to (\mathcal{C}_0,\nabla_0)$, the 2-morphism $\theta$ gives rise to the morphism $\gr(\theta) \colon \widetilde{\gr}_{\mathcal{C}_1}\nabla_1(i\mathcal{C}_1) \to \widetilde{\gr}_{\mathcal{C}_0}\nabla_0(i\mathcal{C}_0)$. In particular, under the identifications \eqref{automorphisms with DES to torsors with curv} and \eqref{torsors with curv to automorphisms with DES}, the assignment $\shAut(\mathcal{C},\nabla) \ni (F,\theta) \mapsto \gr(\theta) \in \dfrac1t\mathbb{C}$ defines a functor (with discrete target) which corresponds to the morphism of Picard stacks $(\mathbb{C}[[t]]^\times \xrightarrow{t\partial_t\log}\dfrac1t\mathbb{C}[[t]])[1] \to \dfrac1t\mathbb{C}[1]$. We summarize the foregoing discussion in the following proposition.

\begin{prop}\label{prop: dq with des to antiderivs}
The assignment $\widetilde{\DQ_{X/\mathcal{P}}^\omega} \ni (\mathcal{C},\nabla) \mapsto \widetilde{\gr}_\mathcal{C}\nabla(i\mathcal{C}) \in \int\frac1t\omega$ extends to a morphism of torsors
\begin{equation}\label{dq with des to antiderivs}
\widetilde{\gr} \colon \widetilde{\DQ_{X/\mathcal{P}}^\omega} \to \int\frac1t\omega
\end{equation}
relative to the morphism of Picard stacks $(\mathbb{C}[[t]]^\times \xrightarrow{t\partial_t\log}\dfrac1t\mathbb{C}[[t]])[2] \to \dfrac1t\mathbb{C}[1]$.
\end{prop}

Thus, under the morphism $(\mathbb{C}[[t]]^\times \xrightarrow{t\partial_t\log}\dfrac1t\mathbb{C}[[t]])[3] \to \dfrac1t\mathbb{C}[2]$, the $(\mathbb{C}[[t]]^\times \xrightarrow{t\partial_t\log}\dfrac1t\mathbb{C}[[t]])[1]$-gerbe $\widetilde{\DQ_{X/\mathcal{P}}^\omega}$ is mapped to the $\dfrac1t\mathbb{C}$-gerbe $\int\frac1t\omega$. In particular, the component of the class of $\widetilde{\DQ_{X/\mathcal{P}}^\omega}$ in $H^2(X;\dfrac1t\mathbb{C})$ is equal to the cohomology class of $\frac1t\omega$.

\section{Self-duality structures} \label{section: Self-duality structures}
\subsection{Self-duality}
We denote by $a \colon \mathbb{C}[[t]] \to \mathbb{C}[[t]]$ the automorphism determined by $t \mapsto -t$. Note that $a$ is an involution: $a\circ a = \id$. For a $\mathbb{C}[[t]]$-module $M$ we denote by $aM$ the $\mathbb{C}$-vector space $M$ with the $\mathbb{C}[[t]]$-module structure given by the composition $\mathbb{C}[[t]] \xrightarrow{a} \mathbb{C}[[t]] \to \End_\mathbb{C}(M)$.

The automorphism $\mathbb{C}[[t]] \xrightarrow{a} \mathbb{C}[[t]]$ induces an automorphism $a \colon \DQ_{X/\mathcal{P}} \to \DQ_{X/\mathcal{P}}$. For $\mathcal{C}\in\DQ_{X/\mathcal{P}}$ the DQ-algebroid $a\mathcal{C}$ admits the following description:
\begin{itemize}
\item[\emph{Objects:}] The algebroid $a\mathcal{C}$ has the same objects as $\mathcal{C}$; for an object $L\in\mathcal{C}$ we denote by $aL$ the corresponding object of $a\mathcal{C}$

\item[\emph{Morphisms:}] For $L_0,L_1\in\mathcal{C}$,  $\Hom_{a\mathcal{C}}(aL_1,aL_0) = a\Hom_\mathcal{C}(L_1,L_0)$.
\end{itemize}

The DQ-algebroids $\left(a\mathcal{C}\right)^\op$ and $a\left(\mathcal{C}^\op\right)$ coincide.
We will denote the common value of $\left(a\mathcal{C}\right)^\op$ and $a\left(\mathcal{C}^\op\right)$ by $\mathcal{C}^\dagger$. The assignment $\mathcal{C} \mapsto \mathcal{C}^\dagger$ extends to a morphism $\left(\bullet\right)^\dagger \colon \DQ_{X/\mathcal{P}} \to \DQ_{X/\mathcal{P}}^\co$.\footnote{For a 2-category $\mathfrak{S}$, the 2-category  $\mathfrak{S}^\co$ has same objects and 1-morphisms, with $\Hom_{\mathfrak{S}^\co}(\cdot,\cdot) = \Hom_\mathfrak{S}(\cdot,\cdot)^\op$.} The functor $\left(\bullet\right)^\dagger$ is an involution: $\mathcal{C}^{\dagger\dagger} = \mathcal{C}$. Moreover, if $\mathcal{C} \in \DQ_{X/\mathcal{P}}^\pi$ then $\mathcal{C}^\dagger \in \DQ_{X/\mathcal{P}}^\pi$, i.e. the involution $\left(\bullet\right)^\dagger$ restricts to an involution $\left(\bullet\right)^\dagger \colon \DQ^\pi_{X/\mathcal{P}} \to \DQ^\pi_{X/\mathcal{P}}$.

The algebroid $\mathcal{C}^\dagger$ admits the following description:
\begin{itemize}
\item[\emph{Objects:}] The algebroid $\mathcal{C}^\dagger$ has the same objects as $\mathcal{C}$; for an object $L\in\mathcal{C}$ we denote by $L^\dagger$ the corresponding object of $a\mathcal{C}$

\item[\emph{Morphisms:}] For $L_0,L_1\in\mathcal{C}$,  $\Hom_{\mathcal{C}^\dagger}(L_1^\dagger,L_0^\dagger) = a\Hom_\mathcal{C}(L_0,L_1)$.
\end{itemize}

A 1-morphism $F\colon \mathcal{C}_1 \to \mathcal{C}_0$ induces the 1-morphism $F^\dagger \colon \mathcal{C}_1^\dagger \to \mathcal{C}_0^\dagger$ with $F^\dagger(L^\dagger) = F(L)^\dagger$ and the effect on morphisms given by
\begin{multline*}
\Hom_{\mathcal{C}_0^\dagger}(L_1^\dagger,L_0^\dagger) = a\Hom_{\mathcal{C}_0}(L_0,L_1) \xrightarrow{F} a\Hom_{\mathcal{C}_1}(F(L_0),F(L_1)) \\
= \Hom_{\mathcal{C}_1^\dagger}(F(L_1)^\dagger,F(L_0)^\dagger)
= \Hom_{\mathcal{C}_1^\dagger}(F^\dagger(L_1^\dagger),F^\dagger(L_0^\dagger))
\end{multline*}

For $F, G \colon\mathcal{C}_1 \to \mathcal{C}_0$ and $L\in\mathcal{C}_1$,
\[
\Hom_{\mathcal{C}_0}(F^\dagger(L^\dagger),G^\dagger(L^\dagger)) = \Hom_{\mathcal{C}_0}(F(L)^\dagger, G(L)^\dagger) = a\Hom_{\mathcal{C}_0}(G(L),F(L)) .
\]
Hence,
\[
\Hom(F^\dagger,G^\dagger) = a\Hom(G,F),\ \ \text{and}\ \ \Hom_{\DQ_{X/\mathcal{P}}}(\mathcal{C}_1^\dagger, \mathcal{C}_0^\dagger) = a\Hom_{\DQ_{X/\mathcal{P}}}(\mathcal{C}_1, \mathcal{C}_0)^\op .
\]

For a functor $F \colon \mathcal{C} \to \mathcal{C}^\dagger$, the functor $F^\dagger \colon \mathcal{C}^\dagger \to \mathcal{C}^{\dagger\dagger} = \mathcal{C}$ is given on objects by $F^\dagger(L^\dagger) = F(L)$. For $L_0,L_1\in\mathcal{C}$ the map $F^\dagger \colon \Hom_{\mathcal{C}^\dagger}(L_1^\dagger,L_0^\dagger) \to \Hom_\mathcal{C}(F(L_1), F(L_0))$ is the map
$aF \colon a\Hom_\mathcal{C}(L_0,L_1) \to \Hom_\mathcal{C}(F(L_1), F(L_0))$

%\subsection{Self-duality structures}
\begin{definition}
A \emph{self-duality structure} on a DQ-algebroid $\mathcal{C}$ is a 1-morphism $T \colon \mathcal{C} \to \mathcal{C}^\dagger$.
\end{definition}

\subsection{DQ-algebroids with self-duality structures}\label{subsection: DQ-algebroids with self-duality structures}
DQ-algebroids equipped with self-duality structures form a 2-category.
\begin{itemize}
\item[\emph{Objects:}] The objects are pairs $(\mathcal{C}, T)$, where $\mathcal{C}$ is a DQ-algebroid and $T$ is a self-duality structure on $\mathcal{C}$.

\item[\emph{1-morphisms:}]
Suppose that $(\mathcal{C}_i,T_i)$, $i=0,1$, are DQ-algebroids equipped with transposition structures. A 1-morphism $(F,\xi) \colon (\mathcal{C}_1,T_1) \to (\mathcal{C}_0,T_0)$ is a pair $(F,\xi)$, where $F\colon \mathcal{C}_1 \to \mathcal{C}_0$ is a 1-morphism and $\xi \colon T_0\circ F \to F^\dagger\circ T_1$ is a 2-morphism between 1-morphisms $\mathcal{C}_1 \to \mathcal{C}_0^\dagger$.

\item[\emph{Composition:}]
Suppose that $(\mathcal{C}_2,T_2) \xrightarrow{(G,\eta)} (\mathcal{C}_1,T_1) \xrightarrow{(F,\xi)} (\mathcal{C}_0,T_0)$ are 1-morphisms. The composition
\begin{equation}\label{transpostion 1-morphism composition}
T_0\circ F\circ G  \xrightarrow{\xi\diamond\id_G} F^\dagger\circ T_1\circ G \xrightarrow{\id_{F^\dagger}\diamond\eta} F^\dagger\circ G^\dagger\circ T_2 = (F\circ G)^\dagger\circ T_2
\end{equation}
gives rise to the 1-morphism $(F\circ G,\eqref{transpostion 1-morphism composition}) \colon (\mathcal{C}_2,T_2) \to (\mathcal{C}_0,T_0)$ defined to be the composition $(F,\xi) \circ (G,\eta)$.

\item[\emph{Identity:}]
The identity 1-morphism $\id_{(\mathcal{C},T)}$ is given by the pair $(\id_\mathcal{C},\id_{\id_T})$.

\item[\emph{2-morphisms:}]
Suppose that $(F_i,\xi_i)$, $i=0,1$, are 1-morphisms $(\mathcal{C}_1,T_1) \to (\mathcal{C}_0,T_0))$. A 2-morphism $f\colon (F_1,\xi_1) \to (F_0,\xi_0)$ is a 2-morphism $f \colon F_1 \to F_0$ such that
\[
(f^\dagger\diamond\id_{T_1})\circ\xi_0\circ(\id_{T_0}\diamond f) = \xi_1
\]

\item[\emph{Composition:}]
Horizontal and vertical compositions of 2-morphisms are defined as those in the 2-category of categories.
\end{itemize}

We denote by $\DQ_{X/\mathcal{P}}^{\omega,\mathtt{t}}$ the 2-stack with locally defined objects the pairs $(\mathcal{C}, T)$ with $\mathcal{C}\in\DQ_{X/\mathcal{P}}^\omega$ and 1- and 2-morphisms as above.

\begin{example}\label{example:self-duality and MW}
For a symplectic DQ-algebra $\mathbb{A}$, the DQ-algebroids $(\mathbb{A}^\dagger)^+$ and $(\mathbb{A}^+)^\dagger$ coincide. Thus, a morphism of DQ-algebras $\mathbb{A} \xrightarrow{\phi} \mathbb{A}^\dagger$ induces a self-duality structure on $\mathbb{A}^+$ via $\mathbb{A}^+ \xrightarrow{\phi^+} (\mathbb{A}^\dagger)^+ = (\mathbb{A}^+)^\dagger$.

Let $\mathbb{A}_{MW} = (\mathcal{O}_{X/\mathcal{P}}[[t]], \star)$ be the Moyal Weyl star-product of Example \ref{example: MW}. It is easy to verify that the identity map of $\mathcal{O}_{X/\mathcal{P}}[[t]]$ is an isomorphism $\mathbb{A}_{MW} \to \mathbb{A}_{MW}^\dagger$ and we shall identify $\mathbb{A}_{MW}^\dagger$ with $\mathbb{A}_{MW}$ from now on.

Hence, $(\mathbb{A}_{MW}^+)^\dagger = \mathbb{A}_{MW}^+$ and the identity morphism is a self-duality structure on $\mathbb{A}_{MW}^+$ which we denote by $T_{can}$.

In particular, any morphism $\phi\colon \mathbb{A}_{MW} \to \mathbb{A}_{MW}^{\dagger}$ is an automorphism of $\mathbb{A}_{MW}$. By Corollary \ref{cor: props of DQ alg} any automorphism of $\mathbb{A}_{MW}$ is inner. Therefore, any self-duality structure on $\mathbb{A}_{MW}^+$ is isomorphic to $T_{can}$.
\end{example}

\subsection{}
Suppose that $T$ is a self-duality structure on a symplectic DQ-algebroid $\mathcal{C}$ and $(F,\xi) \in \Aut_{\DQ_{X/\mathcal{P}}^{\omega,\mathtt{t}}}(\mathcal{C},T)$.

For $\alpha, \beta\in\shHom_{\shAut(\mathcal{C})}(\id_\mathcal{C},F)$ let $\langle \alpha, \beta\rangle_\xi \colon T \to T$ denote the composition
\[
T = T\circ\id_\mathcal{C} \xrightarrow{\id_T\diamond \alpha} T\circ F \xrightarrow{\xi} F^\dagger\circ T \xrightarrow{\beta^\dagger\diamond\id_T} \id_\mathcal{C}\circ T = T
\]
The assignment $(\alpha, \beta) \mapsto \langle \alpha, \beta\rangle_\xi$ defines a ``hermitian'' pairing
\begin{equation}\label{hermitian pairing}
\langle\cdot,\cdot\rangle_{(F,\xi)} \colon \shHom_{\shAut(\mathcal{C})}(\id_\mathcal{C},F)\times a\shHom_{\shAut(\mathcal{C})}(\id_\mathcal{C},F) \to \shEnd_{\shHom(\mathcal{C},\mathcal{C}^\dagger)}(T) = \mathbb{C}[[t]]^\times \ .
\end{equation}
Let
\begin{equation}\label{quadratic form}
\mathtt{q}_{(F,\xi)} \colon \shHom_{\shAut(\mathcal{C})}(\id_\mathcal{C},F) \to \mathbb{C}[[t]]^\times
\end{equation}
denote the  associated quadratic form given by $\mathtt{q}_{(F,\xi)}(\alpha) = \langle \alpha, \alpha\rangle_{(F,\xi)}$.

Let $\mathtt{q}_0 \colon \mathbb{C}[[t]]^\times \to \mathbb{C}[[t]]^\times$ denote the map $f(t) \mapsto f(t)\cdot f(-t)$.

\lemma
The map \eqref{quadratic form} satisfies $\mathtt{q}_{(F,\xi)}(f\cdot\alpha) = \mathtt{q}_0(f)\cdot\mathtt{q}_{(F,\xi)}(\alpha)$.
\endlemma
\begin{proof}
Since the pairing \eqref{hermitian pairing} satisfies $\langle f(t)\cdot\alpha, g(t)\cdot\beta\rangle_{(F,\xi)} = f(t)\cdot g(-t)\cdot\langle \alpha, \beta\rangle_{(F,\xi)}$, the associated quadratic form satisfies $\mathtt{q}_{(F,\xi)}(f(t)\cdot\alpha) = f(t)\cdot f(-t)\cdot\mathtt{q}_{(F,\xi)}(\alpha) = \mathtt{q}_0(f(t)) \cdot \mathtt{q}_{(F,\xi)}(\alpha)$.
\end{proof}
\begin{cor}
The pair $(\shHom_{\shAut(\mathcal{C})}(\id_\mathcal{C},F), \mathtt{q}_{(F,\xi)})$ is a $(\mathbb{C}[[t]]^\times \xrightarrow{\mathtt{q}_0} \mathbb{C}[[t]]^\times)$-torsor.
\end{cor}

The assignment $(F,\xi) \mapsto (\shHom_{\shAut(\mathcal{C})}(\id_\mathcal{C},F), \mathtt{q}_{(F,\xi)}i)$ extends to a morphism
\begin{equation}\label{aut to torsors with q}
\shAut_{\DQ_{X/\mathcal{P}}^{\omega,\mathtt{t}}}(\mathcal{C},T) \to (\mathbb{C}[[t]]^\times \xrightarrow{\mathtt{q}_0} \mathbb{C}[[t]]^\times)[1].
\end{equation}
The canonical morphism $\mathbb{C}[[t]]^\times[1] \to \shAut_{\DQ_{X/\mathcal{P}}^{\omega}}(\mathcal{C})$ lifts to the morphism
\begin{equation}\label{torsors with q to aut with sd}
(\mathbb{C}[[t]]^\times \xrightarrow{\mathtt{q}_0} \mathbb{C}[[t]]^\times)[1] \to \shAut_{\DQ_{X/\mathcal{P}}^{\omega,\mathtt{t}}}(\mathcal{C},T) \colon (S,\mathtt{q}) \mapsto (S\otimes(\cdot), \mathtt{q}).
\end{equation}

\begin{prop}\label{prop: algds with self-duality 2-torsor}
{\ }
\begin{enumerate}
\item The morphism \eqref{aut to torsors with q} and the morphism \eqref{torsors with q to aut with sd} are mutually quasi-inverse monoidal equivalences.

\item The 2-stack $\DQ_{X/\mathcal{P}}^{\omega,\mathtt{t}}$ is a $(\mathbb{C}[[t]]^\times \xrightarrow{\mathtt{q}_0} \mathbb{C}[[t]]^\times)[1]$-gerbe.
\end{enumerate}
\end{prop}
\begin{proof}
It is clear that the morphisms \eqref{aut to torsors with q} and \eqref{torsors with q to aut with sd} are mutually quasi-inverse monoidal equivalences. In view of the equivalence $\shAut(\mathcal{C},T)) \cong (\mathbb{C}[[t]]^\times \xrightarrow{\mathtt{q}_0} \mathbb{C}[[t]]^\times)[1]$ it remains to show that $\DQ_{X/\mathcal{P}}^{\omega,\mathtt{t}}$ is locally non-empty and locally connected. 

Since, locally, every DQ-algebroid is equivalent to one of the form $\mathbb{A}_{MW}^+$ of Example \ref{example:self-duality and MW}, it follows that $\DQ_{X/\mathcal{P}}^{\omega,\mathtt{t}}$ is locally non-empty and locally connected.
\end{proof}

\subsection{Classical limits of self-duality structures}\label{subsection: Classical limits of self-duality structures}
We denote by $\gr^\times\mathcal{C}$ the stack associated to the prestack $i(\mathcal{C}/t\mathcal{C})$.

Recall that, according to \cite{QCL}, the classical limit functor
\[
\gr^\times \colon \DQ_{X/\mathcal{P}}^\omega \to \mathcal{O}_{X/\mathcal{P}}^\times[2] \colon \mathcal{C} \mapsto \gr^\times\mathcal{C},
\]
admits a canonical lifting
\[
\widetilde{\gr^\times} \colon \DQ_{X/\mathcal{P}}^\omega \to (\mathcal{O}_{X/\mathcal{P}}^\times \xrightarrow{d\log} \Omega_{X/\mathcal{P}}^1 \xrightarrow{d} \Omega_{X/\mathcal{P}}^{2,cl})[2]
\]
Since the map $\mathbb{C}^\times \to (\mathcal{O}^\times_{X/\mathcal{P}}\to\Omega^1_{X/\mathcal{P}}\to\Omega^{2,cl}_{X/\mathcal{P}})$ is a quasiisomorphism, we will regard the morphism $\widetilde{\gr^\times}$ as taking values in $\mathbb{C}^\times[2]$.

Suppose that $T$ is a self-duality structure on a symplectic DQ-algebroid $\mathcal{C} \in \DQ_{X/\mathcal{P}}^\omega$. The quasi-classical limit of $T$ is an equivalence
\[
\widetilde{\gr^\times}(T) \colon \widetilde{\gr^\times}\mathcal{C} \to \widetilde{\gr^\times}(\mathcal{C}^\dagger) = (\widetilde{\gr^\times}\mathcal{C})^\op = (\widetilde{\gr^\times}\mathcal{C})^{\otimes -1}
\]
The equivalence $\widetilde{\gr^\times}\mathcal{C} \xrightarrow{\widetilde{\gr^\times}(T)} (\widetilde{\gr^\times}\mathcal{C})^{\otimes -1}$ determines a structure of a $\mu_2$-gerbe on $\widetilde{\gr^\times}\mathcal{C}$, where $\mu_2 := \ker(\mathbb{C}^\times\xrightarrow{\lambda\mapsto\lambda^2}\mathbb{C}^\times)$ is the group of square roots of one. Denoting this structure by $\widetilde{\gr^\times}(\mathcal{C},T)$ we obtain the commutative diagram
\[
\begin{CD}
\DQ_{X/\mathcal{P}}^{\omega,\mathtt{t}} @>{\widetilde{\gr^\times}}>> \mu_2[2] \\
@VVV @VVV \\
\DQ_{X/\mathcal{P}}^\omega @>{\widetilde{\gr^\times}}>> \mathbb{C}^\times[2]
\end{CD}
\]

\section{Compatibility of DES with self-duality structures} \label{section: Compatibility of DES with self-duality structures}
\subsection{Duality and derivations}
Suppose that $A$ is an associative $\mathbb{C}$-algebra. We denote by $A^{Lie}$ the Lie algebra structure on $A$ given by the commutator bracket.
\begin{itemize}
\item Since the algebras $A$ and $A^\op$ share the underlying vector space, it follows that $\End_\mathbb{C}(A) = \End_\mathbb{C}(A^\op)$. Under this identification the subspaces $\Der_\mathbb{C}(A)$ and $\Der_\mathbb{C}(A^\op)$ coincide. Thus, $\Der_\mathbb{C}(A)=\Der_\mathbb{C}(A^\op)$ as Lie algebras.

\item The identity map is an isomorphism of Lie algebras $(A^{Lie})^\op = (A^\op)^{Lie}$.

\item Under these identifications the map of Lie algebras $\ad^\op \colon (A^\op)^{Lie} \to \Der_\mathbb{C}(A^\op)$ is given by $-\ad \colon A^{Lie} \to \Der_\mathbb{C}(A)$.
\end{itemize}

Recall that the $\mathbb{C}$-algebra automorphism $a \colon \mathbb{C}[[t]] \to \mathbb{C}[[t]]$ is given by $t \mapsto -t$. Conjugation by $a$ induces the Lie algebra automorphism $\Conj(a) \colon \Der_\mathbb{C}(\mathbb{C}[[t]]) \to \Der_\mathbb{C}(\mathbb{C}[[t]])$. Since $\Conj(a)(t\partial_t) = t\partial_t$, the automorphism $\Conj(a)$ restricts to the identity map on $\mathfrak{l} = \mathbb{C}\cdot t\partial_t$.

Let $A$ be an associative $\mathbb{C}[[t]]$-algebra. The $\mathbb{C}[[t]]$-algebras $(aA)^\op$ and $a(A^\op)$ coincide and we denote their common value by $A^\dagger$.
\begin{itemize}
\item Since the algebras $A$ and $A^\dagger$ share the underlying vector space, it follows that $\End_\mathbb{C}(A) = \End_\mathbb{C}(A^\dagger)$. Under this identification the subspaces $\Der_\mathbb{C}(A)$ and $\Der_\mathbb{C}(A^\dagger)$ coincide. Thus, $\Der_\mathbb{C}(A)=\Der_\mathbb{C}(A^\dagger)$ as Lie algebras.

\item The subspaces $\Der_{\mathbb{C}[[t]]}(A)$ and $\Der_{\mathbb{C}[[t]]}(A^\dagger)$ of $\Der_\mathbb{C}(A)=\Der_\mathbb{C}(A^\dagger)$ coincide. Thus, $\Der_{\mathbb{C}[[t]]}(A)=\Der_{\mathbb{C}[[t]]}(A^\dagger)$ as Lie algebras.
\end{itemize}

Suppose that $\mathbb{A}$ is a symplectic DQ-algebra.
\begin{itemize}
\item The subsheaves $\mathfrak{L}(\mathbb{A})$ and $\mathfrak{L}(\mathbb{A}^\dagger)$ of $\shEnd_\mathbb{C}(\mathbb{A})$ coincide. Thus, $\mathfrak{L}(\mathbb{A}) = \mathfrak{L}(\mathbb{A}^\dagger)$ as Lie algebras.
\item $\dfrac1t\mathbb{A}^\dagger = \dfrac1t\mathbb{A}^\op$ as Lie algebras.
\item The diagram
\[
\begin{CD}
\dfrac1t\mathbb{C}[[t]] @>>> \dfrac1t\mathbb{A} @>{ad}>> \Der_{\mathbb{C}[[t]]}(\mathbb{A}) @>>> \mathfrak{L}(\mathbb{A}) @>>> \mathfrak{l} \\
@V{-\id}VV @V{-\id}VV @| @| @| \\
\dfrac1t\mathbb{C}[[t]] @>>> \dfrac1t\mathbb{A^\dagger} @>{ad}>> \Der_{\mathbb{C}[[t]]}(\mathbb{A}^\dagger) @>>> \mathfrak{L}(\mathbb{A}^\dagger) @>>> \mathfrak{l}
\end{CD}
\]
is commutative.
\end{itemize}

\subsection{Duality and DES}
Suppose that $\mathbb{A}$ is a symplectic DQ-algebra. For $\widetilde{\mathfrak{L}} \in \DES(\mathbb{A})$, the sheaf $\widetilde{\mathfrak{L}}^\dagger$ is defined by the push-out diagram
\begin{equation}\label{push out des transpose}
\begin{CD}
\dfrac1t\mathbb{A} @>>> \widetilde{\mathfrak{L}} \\
@V{-\id}VV @VVV \\
\dfrac1t\mathbb{A^\dagger} @>>> \widetilde{\mathfrak{L}}^\dagger
\end{CD}
\end{equation}
The assignment $\widetilde{\mathfrak{L}} \mapsto \widetilde{\mathfrak{L}}^\dagger$ extends to a functor
\[
(\cdot)^\dagger\colon\DES(\mathbb{A}) \to \DES(\mathbb{A}^\dagger) .
\]
Moreover, for a morphism $\phi\colon\mathbb{A}_1 \to \mathbb{A}_0$ of symplectic DQ-algebras the diagram
\[
\begin{CD}
\DES(\mathbb{A}_1) @>{(\cdot)^\dagger}>> \DES(\mathbb{A}_1^\dagger) \\
@V{\phi_*}VV @VV{\phi_*^\dagger}V \\
\DES(\mathbb{A}_0) @>{(\cdot)^\dagger}>> \DES(\mathbb{A}_0^\dagger) 
\end{CD}
\]
is commutative.

Suppose that $\mathcal{C}$ is a symplectic DQ-algebroid. For $L\in\mathcal{C}$, $\widetilde{\mathfrak{L}} \in \DES(\shEnd_\mathcal{C}(L))$ let $(L,\widetilde{\mathfrak{L}})^\dagger := (L^\dagger, \widetilde{\mathfrak{L}}^\dagger)$.

For $\nabla \in \DES(\mathcal{C})$, the assignment $L^\dagger \mapsto \nabla(L)^\dagger$ defines a DES on $\mathcal{C}^\dagger$ which we denote by $\nabla^\dagger$. The assignment $\nabla \mapsto \nabla^\dagger$ extends to a morphism of stacks
\[
(\cdot)^\dagger\colon\shDES(\mathcal{C}) \to \shDES(\mathcal{C}^\dagger) .
\]

\definition
Let $\mathcal{C}$ be a symplectic DQ-algebroid, let $T$ be a self-duality structure on $\mathcal{C}$, and let $\nabla \in \DES(\mathcal{C})$.

A \emph{compatibility} between $T$ and $\nabla$ is a 2-morphism $\tau \colon \widetilde{T}\circ\nabla \to \nabla^\dagger\circ T$ such that the pair $(T,\tau)$ is a 1-morphism $(\mathcal{C},\nabla) \to (\mathcal{C}^\dagger,\nabla^\dagger)$ of DQ-algebroids with DES (see \ref{subsection: Symplectic DQ-algebroids with DES}).
\enddefinition

\subsection{Symplectic DQ-algebroids with compatible self-duality and DES}
Symplectic DQ-algebroids equipped with compatible self-duality and dilation-equivariance structures form a 2-category in the following manner.
\begin{itemize}
\item[\emph{Objects:}] The objects are quadruples $(\mathcal{C}, \nabla, T, \tau)$, where $\mathcal{C}$ is a symplectic DQ-algebroid, $\nabla$ is a DES on $\mathcal{C}$, $T$ is a self-duality structure on $\mathcal{C}$, and $\tau$ is a compatibility between $T$ and $\nabla$.

\item[\emph{1-morphisms:}] Suppose that $(\mathcal{C}_i, \nabla_i, T_i, \tau_i)$, $i = 0,1$, are objects as above. A 1-morphism $(\mathcal{C}_1, \nabla_1, T_1, \tau_1) \to (\mathcal{C}_0, \nabla_0, T_0, \tau_0)$ is a triple $(F,\theta,\xi)$ such that  $(F,\theta)$ is a 1-morphism $(\mathcal{C}_1, \nabla_1) \to (\mathcal{C}_0, \nabla_0)$ of DQ-algebroids with DES and $(F,\xi)$ is a 1-morphism $(\mathcal{C}_1, T_1) \to (\mathcal{C}_0, T_0)$ of DQ-algebroids with a self-duality structures (see \ref{subsection: DQ-algebroids with self-duality structures}) subject to the commutativity condition
\begin{equation}\label{eqn:condition-compatibility}
	\begin{tikzcd}
		& \widetilde{T}_{0} \circ \nabla_{0} \circ F \arrow[dr, "\tau_{0} \diamond \id{}_{F}"] & \\
		\widetilde{T}_{0} \circ \widetilde{F} \circ \nabla_{1} \arrow[ur, "\id{}_{\widetilde{T}_{0}} \diamond \theta"] \arrow[d, "\widetilde{\xi}\diamond \id{}_{\nabla_{1}}"] & & \nabla_{0}^{\dagger} \circ T_{0} \circ F \arrow[d, "\id{}_{\nabla_{0}^{\dagger}} \diamond \xi"] \\
		\widetilde{F}^{\dagger} \circ \widetilde{T}_{1} \circ \nabla_{1} \arrow[dr, "\id{}_{\widetilde{F}^{\dagger}} \diamond \tau_{1}"] & & \nabla_{0}^{\dagger} \circ F^{\dagger} \circ T_{1} \\
		& \widetilde{F}^{\dagger} \circ \nabla_{1}^{\dagger} \circ T_{1} \arrow[ur, "\theta^{\dagger} \diamond \id{}_{T_{1}}"] &
	\end{tikzcd}
\end{equation}

\item[\emph{Composition:}] Suppose that $(\mathcal{C}_2, \nabla_2, T_2, \tau_2) \xrightarrow{(G,\rho,\eta)} (\mathcal{C}_1, \nabla_1, T_1, \tau_1) \xrightarrow{(F,\theta,\xi)} (\mathcal{C}_0, \nabla_0, T_0, \tau_0)$ are 1-morphisms. The composition $(F,\theta,\xi)\circ(G,\rho,\eta)$ is given by $(F\circ G, \eqref{composition 1-morphism des}, \eqref{transpostion 1-morphism composition})$.

\item[\emph{Identity:}] The identity automorphism of $(\mathcal{C}, \nabla, T, \tau)$ is given by the triple $(\id_\mathcal{C},\id_\nabla,\id_T)$.

\item[\emph{2-morphisms:}] Suppose that $(F_i,\theta_i,\xi_i)$, $i = 0,1$, are 1-morphisms $(\mathcal{C}_1, \nabla_1, T_1, \tau_1) \to (\mathcal{C}_0, \nabla_0, T_0, \tau_0)$. A 2-morphism $f\colon (F_1,\theta_1,\xi_1) \to (F_0,\theta_0,\xi_0)$ is a 2-morphism $f\colon F_1 \to F_0$ such that $f\colon (F_1,\theta_1) \to (F_0,\theta_0)$ is a 2-morphism between 1-morphisms $(\mathcal{C}_1, \nabla_1) \to (\mathcal{C}_0, \nabla_0)$ of DQ-algebroids with DES and $f\colon (F_1,\xi_1) \to (F_0,\xi_0)$ is a 2-morphism between 1-morphisms $(\mathcal{C}_1, T_1) \to (\mathcal{C}_0, T_0)$ of DQ-algebroids with a self-duality structures.

\item[\emph{Composition:}] Horizontal and vertical composition of 2-morphisms are defined
as in the 2-category of all categories.
\end{itemize}
We denote by $\widetilde{\DQ_{X/\mathcal{P}}^{\omega,\mathtt{t}}}$ the 2-stack with locally defined objects $(\mathcal{C}, \nabla, T, \tau)$ with $\mathcal{C}_i\in\DQ_{X/\mathcal{P}}^\omega$, and 1- and 2-morphisms as above.

\begin{example}\label{example: compatibility MW}
Let $\widetilde{\mathfrak{L}}$ be a DES on $\mathbb{A}_{MW}$ associated to a splitting of \eqref{ses extended derivations}. Let $\nabla$ denote the associated DES on $\mathbb{A}_{MW}^{+}$ as described in Example \ref{example: DES on MW}. The identification $(\mathbb{A}_{MW}^{+})^{\dagger} = \mathbb{A}_{MW}^{+}$, induces the identifications $\widetilde{(\mathbb{A}_{MW}^{+})^{\dagger}} = \widetilde{\mathbb{A}_{MW}^{+}}$, $T_{can} = \id$, $\widetilde{T_{can}} = \id$ and $\nabla^\dagger = \nabla$. Then, the 2-morphism $\id \colon \nabla \to \nabla$ is a compatibility between $T_{can}$ and $\nabla$.

In fact any two quadruples $Q_i := (\mathbb{A}_{MW}^{+}, T_i, \nabla_i, \tau_i)$, $i=0,1$, are isomorphic. Without loss of generality we may assume that $T_0=T_1=T_{can}$ and $\nabla_0=\nabla_1$ are associated to a split DES  $\widetilde{\mathfrak{L}} = \dfrac1t\mathbb{A}_{MW}\oplus\mathfrak{l}$ as above. If $(\id, \theta, \id) \colon Q_1 \to Q_0$ is a 1-morphism, the condition \eqref{eqn:condition-compatibility} reduces to commutativity of the diagram of morphisms of DES
\begin{equation}\label{compatibilities on MW}
\begin{CD}
\widetilde{\mathfrak{L}} @>{\theta}>> \widetilde{\mathfrak{L}} \\
@V{\tau_1}VV @VV{\tau_0}V \\
\widetilde{\mathfrak{L}}^{\dagger} @>{\theta^{\dagger}}>> \widetilde{\mathfrak{L}}^{\dagger}
\end{CD}
\end{equation}
In terms of the identifications $\widetilde{\mathfrak{L}} = \dfrac1t\mathbb{A}_{MW}\oplus\mathfrak{l}$ and $\widetilde{\mathfrak{L}}^\dagger = \dfrac1t\mathbb{A}_{MW}\oplus\mathfrak{l}$ given by the induced splitting, the isomorphism $\widetilde{\mathfrak{L}} \xrightarrow{\cong} \widetilde{\mathfrak{L}}^\dagger$ in \eqref{push out des transpose} is given by $\begin{bmatrix} -1 & 0 \\ 0 & 1 \end{bmatrix}$. Therefore, if $\theta\colon \widetilde{\mathfrak{L}} \to \widetilde{\mathfrak{L}}$ is represented by $\begin{bmatrix} 1 & \phi \\ 0 & 1 \end{bmatrix}$, then $\theta^\dagger\colon \widetilde{\mathfrak{L}}^\dagger \to \widetilde{\mathfrak{L}}^\dagger$ is represented by $\begin{bmatrix} 1 & -\phi \\ 0 & 1 \end{bmatrix}$.

If $\tau_i$ is represented by $\begin{bmatrix} 1 & \psi_i \\ 0 & 1 \end{bmatrix}$, the condition \eqref{compatibilities on MW} is equivalent to $\psi_0 + \phi = \psi_1 = \phi$, which is to say $\phi = \dfrac12(\psi_1-\psi_0)$.
\end{example}

\subsection{}
Suppose that $(\mathcal{C}, \nabla, T, \tau)\in \widetilde{\DQ_{X/\mathcal{P}}^{\omega,\mathtt{t}}}$.

Since, for $\nu\in\shEnd_{\Hom(\mathcal{C},\mathcal{C}^\dagger)}(T)$,
\[
\id_{p_{\mathcal{C}^\dagger}}\diamond\tau = \id_{p_{\mathcal{C}^\dagger}}\diamond(\tau\circ(\widetilde{\nu}\diamond\id_\nabla))
\]
it follows that the map
\begin{equation}\label{compatibility curv}
\shEnd_{\Hom(\mathcal{C},\mathcal{C}^\dagger)}(T) \to \shHom_{\shHom(i\mathcal{C},\widetilde{\mathcal{C}^\dagger})}(\widetilde{T}\circ\nabla, \nabla^\dagger\circ T) \colon \nu \mapsto \tau - \tau\circ(\widetilde{\nu}\diamond\id_\nabla)
\end{equation}
factors through the canonical map $\dfrac1t\mathbb{C}[[t]] \to \shHom_{\shHom(i\mathcal{C},\widetilde{\mathcal{C}^\dagger})}(\widetilde{T}\circ\nabla, \nabla^\dagger\circ T)$. Moreover, the diagram
\[
% \begin{CD}
% \mathbb{C}[[t]]^\times @>{t\partial_t\log}>> \dfrac1t\mathbb{C}[[t]] \\
% @V{\cong}VV @VVV \\
% \shEnd_{\shHom(\mathcal{C},\mathcal{C}^\dagger)}(T) @>{\eqref{compatibility curv}}>> \shHom_{\shHom(i\mathcal{C},\widetilde{\mathcal{C}^\dagger})}(\widetilde{T}\circ\nabla, \nabla^\dagger\circ T)
% \end{CD}
\begin{tikzcd}
	\mathbb{C}[[t]]^\times \arrow[r, "t\partial_t\log"] \arrow[d, "\simeq"] & \dfrac1t\mathbb{C}[[t]] \arrow[d] \\
	\shEnd_{\shHom(\mathcal{C},\mathcal{C}^\dagger)}(T) \arrow[r, "\eqref{compatibility curv}"] & \shHom_{\shHom(i\mathcal{C},\widetilde{\mathcal{C}^\dagger})}(\widetilde{T}\circ\nabla, \nabla^\dagger\circ T)
\end{tikzcd}
\]
is commutative.

Suppose that $(F,\theta,\xi)$ is an automorphism of $(\mathcal{C}, \nabla, T, \tau)$. Then, 
\begin{equation}\label{q and c together torsors}
% \begin{CD}
% \shHom_{\shAut(\mathcal{C})}(\id_\mathcal{C}, F) @>{c(F,\theta)}>> \dfrac1t\mathbb{C}[[t]] \\
% @V{\mathtt{q}_{(F,\xi)}}VV @VV{\tau\circ\avg}V \\
% \shEnd_{\Hom(\mathcal{C},\mathcal{C}^\dagger)}(T) @>{\eqref{compatibility curv}}>> \shHom_{\shHom(i\mathcal{C},\widetilde{\mathcal{C}^\dagger})}(\widetilde{T}\circ\nabla, \nabla^\dagger\circ T)
% \end{CD}
\begin{tikzcd}
	\shHom_{\shAut(\mathcal{C})}(\id_\mathcal{C}, F) \arrow[r, "{c(F, \theta)}"] \arrow[d, "{\mathtt{q}_{(F,\xi)}}"] & \dfrac1t\mathbb{C}[[t]] \arrow[d, "\tau\circ\avg"] \\
	\shEnd_{\shHom(\mathcal{C},\mathcal{C}^\dagger)}(T) \arrow[r, "\eqref{compatibility curv}"] & \shHom_{\shHom(i\mathcal{C},\widetilde{\mathcal{C}^\dagger})}(\widetilde{T}\circ\nabla, \nabla^\dagger\circ T)
\end{tikzcd}
\end{equation}
where $\avg(f) = f + a(f)$, is a commutative diagram of morphisms of torsors relative to the commutative diagram of morphisms of sheaves of groups
\begin{equation}\label{q and c together groups}
\begin{CD}
\mathbb{C}[[t]]^\times @>{t\partial_t\log}>> \dfrac1t\mathbb{C}[[t]] \\
@V{\mathtt{q}_0}VV @VV{\avg}V \\
\mathbb{C}[[t]]^\times @>{t\partial_t\log}>> \dfrac1t\mathbb{C}[[t]]
\end{CD}
\end{equation}
The commutativity of \eqref{q and c together torsors} and \eqref{q and c together groups} implies that the pair $(\shHom_{\shAut(\mathcal{C})}(\id_\mathcal{C}, F), (\mathtt{q}_{(F,\xi)},\, c(F,\theta)))$ is a torsor under
\begin{multline*}
\tau_{\leqslant 1}(\mathbb{C}[[t]]^\times \xrightarrow{(\mathtt{q}_0,\, t\partial_t\log)} \mathbb{C}[[t]]^\times\oplus\dfrac1t\mathbb{C}[[t]]  \xrightarrow{t\partial_t\log - \avg} \dfrac1t\mathbb{C}[[t]]) \\
= (\mathbb{C}[[t]]^\times \xrightarrow{(\mathtt{q}_0,\, t\partial_t\log)} \ker(t\partial_t\log - \avg))
\end{multline*}
\lemma\label{lemma: cohomology of truncation}
The map of complexes
\[
\left[ \begin{CD}\dfrac1t\mathbb{C} \\ @A{0}AA \\ \mu_2 \end{CD} \right] \longrightarrow \left[ \begin{CD}\ker(t\partial_t\log - \avg)) \\ @AA{(\mathtt{q}_0,\, t\partial_t\log)}A \\ \mathbb{C}[[t]]^\times \end{CD} \right]
\]
with components given by $\dfrac1t\mathbb{C} \xrightarrow{(1,\,\frac1t\mathbb{C}\hookrightarrow \frac1t\mathbb{C}[[t]])} \ker(t\partial_t\log - \avg))$ and the inclusion $\mu_2 \hookrightarrow \mathbb{C}[[t]]^\times$
is a quasiisomorphism.
\endlemma
\begin{proof}
We leave it to the reader to show that $ker(\mathtt{q}_0,\, t\partial_t\log) = \mu_2$.

Let $f\in\mathbb{C}[[t]]^\times$, $g\in\dfrac1t\mathbb{C}[[t]]$. Then, $(f,g)\in \ker(t\partial_t\log - \avg))$ if and only if $t\partial_t\log f = \avg(g)$.

Suppose that $(f,g)\in \ker(t\partial_t\log - \avg))$. Since $\im(t\partial_t\log) = t\mathbb{C}[[t]]$, we may assume that $g = \dfrac{a}t + b \in\dfrac1t\mathbb{C}\oplus\mathbb{C}$ so that $t\partial_t\log(f) = \avg(g) = 2b$ is constant which must be equal to zero since $t\partial_t\log(f) \in t\mathbb{C}[[t]]$. This implies that $f$ is constant, i.e. $f\in\mathbb{C}^\times$ and, hence, $(f,g)$ is cohomologous to a pair of the form $(1,\dfrac1t\lambda)$, $\lambda\in\mathbb{C}$.
\end{proof}

%It follows from Lemma \ref{lemma: cohomology of truncation} that the pair $(\shHom_{\shAut(\mathcal{C})}(\id_\mathcal{C}, F), (\mathtt{q}_0,\, t\partial_t\log))$ is a $(\mu_2 \xrightarrow{0} \dfrac1t\mathbb{C})$-torsor.

The assignment $(F,\theta,\xi) \mapsto (\shHom_{\shAut(\mathcal{C})}(\id_\mathcal{C}, F),\, (\mathtt{q}_0,\, t\partial_t\log))$ extends to a morphism
\begin{equation}\label{aut with des sd comp to torsors}
\shAut_{\widetilde{\DQ_{X/\mathcal{P}}^{\omega,\mathtt{t}}}}(\mathcal{C}, \nabla, T, \tau) \to (\mathbb{C}[[t]]^\times \xrightarrow{(\mathtt{q}_0,\, t\partial_t\log)} \ker(t\partial_t\log - \avg))[1]
\end{equation}
The canonical morphism $\mathbb{C}[[t]]^\times \to \shAut_{\DQ_{X/\mathcal{P}}^\omega}(\mathcal{C})$ lifts to the morphism
\begin{equation}\label{torsors to aut with des sd comp}
(\mathbb{C}[[t]]^\times \xrightarrow{(\mathtt{q}_0,\, t\partial_t\log)} \ker(t\partial_t\log - \avg))[1] \to \shAut_{\widetilde{\DQ_{X/\mathcal{P}}^{\omega,\mathtt{t}}}}(\mathcal{C}, \nabla, T, \tau) 
\end{equation}
given by
\[
(S,(\mathtt{q},c)) \mapsto (S\otimes(\cdot), c, \mathtt{q})
\]

\begin{prop}\label{prop: dq with des sd comp}
{\ }
\begin{enumerate}
\item The morphism \eqref{aut with des sd comp to torsors} and the morphism \eqref{torsors to aut with des sd comp} are mutually quasi-inverse monoidal equivalences.

\item The 2-stack $\widetilde{\DQ_{X/\mathcal{P}}^{\omega,\mathtt{t}}}$ is a $(\mathbb{C}[[t]]^\times \xrightarrow{(\mathtt{q}_0,\, t\partial_t\log)} \ker(t\partial_t\log - \avg))[1]$-gerbe.
\end{enumerate}
\end{prop}
\begin{proof}
It is clear that \eqref{aut with des sd comp to torsors} and the morphism \eqref{torsors to aut with des sd comp} are mutually quasi-inverse monoidal equivalences.

Example \ref{example: compatibility MW} shows that $\widetilde{\DQ_{X/\mathcal{P}}^{\omega,\mathtt{t}}}$ is locally non-empty and locally connected and therefore a $\shAut_{\widetilde{\DQ_{X/\mathcal{P}}^{\omega,\mathtt{t}}}}(\mathcal{C}, \nabla, T, \tau) \cong (\mathbb{C}[[t]]^\times \xrightarrow{(\mathtt{q}_0,\, t\partial_t\log)} \ker(t\partial_t\log - \avg))[1]$-gerbe.
\end{proof}

\section{Classification of symplectic DQ-algebroids} \label{section: Classification of symplectic DQ-algebroids}
In what follows we work with a fixed non-degenerate Poisson bi-vector $\pi\in\Gamma(X;\bigwedge^2\mathcal{T}_{X/\mathcal{P}})$ and set $\omega = \pi^{-1}$.

\subsection{Canonical quantization}
From Proposition \ref{prop: dq with des to antiderivs} and Proposition \ref{prop: dq with des sd comp} we know that the morphism of quasi-classical limit (see \ref{subsection: Classical limits of self-duality structures}) and the morphism \eqref{dq with des to antiderivs} give rise to the equivalence
\begin{equation}\label{DQ with all struct qcl}
\widetilde{\DQ_{X/\mathcal{P}}^{\omega,\mathtt{t}}} \xrightarrow{(\widetilde{\gr^\times},\ \eqref{dq with des to antiderivs})} \mu_2[2] \times\int\frac1t\omega
\end{equation}
which makes the diagram
\[
\begin{CD}
\widetilde{\DQ_{X/\mathcal{P}}^{\omega,\mathtt{t}}} @>{(\widetilde{\gr^\times},\ \eqref{dq with des to antiderivs})}>> \mu_2[2] \times\int\frac1t\omega \\
@VVV @VVV \\
\DQ_{X/\mathcal{P}}^\omega @>{\widetilde{\gr^\times}}>> \mathbb{C}^\times[2]
\end{CD}
\]
commutative.

\begin{thm}\label{thm: canonical quantization}
Every $\mu_2$-gerbe on $X$ admits a canonical quantization.
\end{thm}
\begin{proof}
Let $\widetilde{\DQ_{X/\mathcal{P}}^{\omega,\mathtt{t}}}/\frac1t\mathbb{C}[1]$ denote the push-out of $\widetilde{\DQ_{X/\mathcal{P}}^{\omega,\mathtt{t}}}$ along the morphism of complexes
\[
\left[ \begin{CD}\ker(t\partial_t\log - \avg)) \\ @AA{(\mathtt{q}_0,\, t\partial_t\log)}A \\ \mathbb{C}[[t]]^\times \end{CD} \right]
\longrightarrow
\left[ \begin{CD}\ker(t\partial_t\log - \avg))/\dfrac1t\mathbb{C} \\ @AA{(\mathtt{q}_0,\, t\partial_t\log)}A \\ \mathbb{C}[[t]]^\times \end{CD} \right]
\]
It follows from Lemma \ref{lemma: cohomology of truncation} that $\widetilde{\DQ_{X/\mathcal{P}}^{\omega,\mathtt{t}}}/\frac1t\mathbb{C}[1]$ is a $\mu_2[2]$-torsor. The equivalence \eqref{DQ with all struct qcl} induces the equivalence $\widetilde{\DQ_{X/\mathcal{P}}^{\omega,\mathtt{t}}}/\frac1t\mathbb{C}[1]\xrightarrow{\widetilde{\gr^\times}} \mu_2[2]$
which makes the diagram
\[
\begin{CD}
\widetilde{\DQ_{X/\mathcal{P}}^{\omega,\mathtt{t}}}/\frac1t\mathbb{C}[1] @>{\widetilde{\gr^\times}}>> \mu_2[2] \\
@VVV @VVV \\
\DQ_{X/\mathcal{P}}^\omega @>{\widetilde{\gr^\times}}>> \mathbb{C}^\times[2]
\end{CD}
\]
commutative. Hence, there are morphisms
\[
\DQ_{X/\mathcal{P}}^\omega(X) \longleftarrow \widetilde{\DQ_{X/\mathcal{P}}^{\omega,\mathtt{t}}}/\textstyle{\frac1t}\mathbb{C}[1](X) \xrightarrow{\widetilde{\gr^\times}} \mu_2[2](X)
\]
where $\widetilde{\gr^\times}$ is an equivalence.
\end{proof}

We denote by $\mathcal{C}_{can} \in \DQ_{X/\mathcal{P}}^\omega(X)$ the image under the morphism $\widetilde{\DQ_{X/\mathcal{P}}^{\omega,\mathtt{t}}}/\frac1t\mathbb{C}[1](X) \to \DQ_{X/\mathcal{P}}^\omega(X)$ of the object of $\widetilde{\DQ_{X/\mathcal{P}}^{\omega,\mathtt{t}}}/\frac1t\mathbb{C}[1](X)$ which corresponds to $\mu_2[1] \in \mu_2[2](X)$. 

In what follows we refer to $\mathcal{C}_{can}$ as the \emph{canonical quantization}.

\subsection{Classification of symplectic DQ-algebroids}
For $\mathcal{C}_0, \mathcal{C}_1 \in \DQ^\omega_{X/\mathcal{P}}$ let $[\mathcal{C}_1 :\mathcal{C}_0] := \shHom_{\DQ^\omega_{X/\mathcal{P}}}(\mathcal{C}_0,\mathcal{C}_1) \in {\mathbb{C}[[t]]}^\times[2]$ by analogy with the difference class of \cite{D}.

Let $\DQ^\omega_{X/\mathcal{P}, 0}$ denote the fiber of $\widetilde{\gr^\times} \colon \DQ^\omega_{X/\mathcal{P}} \to \mathbb{C}^\times[2]$ over $\mathbb{C}^\times[1]$. The objects of $\DQ^\omega_{X/\mathcal{P}, 0}$ are quantizations of the trivial gerbe $\mathcal{O}^\times_{X/\mathcal{P}}$ equipped with the trivial flat structure.

\begin{thm}\label{thm: classification}
The morphisms
\begin{equation}\label{diff class equiv}
[\,\bullet\, :\mathcal{C}_{can}] \colon \DQ^\omega_{X/\mathcal{P}} \to {\mathbb{C}[[t]]}^\times[2] \colon \mathcal{C} \mapsto [\mathcal{C} :\mathcal{C}_{can}]
\end{equation}
and
\begin{equation}\label{twisting}
\mathbb{C}^\times[2]\times\DQ^\omega_{X/\mathcal{P}, 0} \to \DQ^\omega_{X/\mathcal{P}}\colon (\mathcal{S},\mathcal{C}) \mapsto \mathcal{S}\otimes\mathcal{C}
\end{equation}
are equivalences.
\end{thm}
\begin{proof}
A quasi-inverse to \eqref{diff class equiv} is given by $\mathcal{S}\mapsto \mathcal{S}\otimes\mathcal{C}_{can}$.

A quasi-inverse to \eqref{twisting} is given by $\mathcal{C} \mapsto (\widetilde{\gr^\times}\mathcal{C}, \widetilde{\gr^\times}\mathcal{C}^{\otimes -1}\otimes\mathcal{C})$.
\end{proof}

\subsection{The Fedosov class}
Under the isomorphism $ {\mathbb{C}[[t]]}^\times = \mathbb{C}^\times \times \exp(t\mathbb{C}[[t]])$ the morphism $\widetilde{\gr^\times}$ coincides with the composition
\[
\DQ^\omega_{X/\mathcal{P}} \xrightarrow{[\,\bullet\, :\mathcal{C}_{can}]} \mathbb{C}[[t]]{}^\times[2] \to \mathbb{C}^\times[2]
\]
Since $\mathcal{C}_{can} \in \DQ^\omega_{X/\mathcal{P}, 0}$, the morphism $[\,\bullet :\mathcal{C}_{can}]$ restricts to the equivalence
\[
[\,\bullet :\mathcal{C}_{can}] \colon \DQ^\omega_{X/\mathcal{P},0} \to \exp(t\mathbb{C}[[t]])[2]
\]
Let $\Phi \colon \DQ^\omega_{X/\mathcal{P},0} \to \mathbb{C}[[t]][2]$ denote the composition
\[
\DQ^\omega_{X/\mathcal{P},0} \xrightarrow{[\,\bullet\, :\mathcal{C}_{can}]} \exp(t\mathbb{C}[[t]])[2] \xrightarrow{\frac1t\log} \mathbb{C}[[t]][2]
\]

The construction of the Fedosov class of a quantization of the structure sheaf was extended to the setting of algebroids in \cite{DQgerbes}. The construction associates to $\mathcal{C} \in  \DQ^\omega_{X/\mathcal{P}, 0}(X)$ the Fedosov class $\Theta(\mathcal{C}) \in \dfrac1t\omega + H^2(X;\mathbb{C}[[t]])$. It is shown in \cite{NT} that $\Theta(\mathcal{C}_{can}) = \dfrac1t\omega$. An argument similar to that in \cite{D} for the case of DQ-algebras shows that for $\mathcal{C}_0, \mathcal{C}_1 \in  \DQ^\omega_{X/\mathcal{P}, 0}(X)$, the class of $\dfrac1t\log([\mathcal{C}_1,\mathcal{C}_0])$ in $H^2(X;\mathbb{C}[[t]])$ coincides with  $\Theta(\mathcal{C}_1) - \Theta(\mathcal{C}_0)$.

Since $\Phi(\mathcal{C}_{can}) = 0$ and, clearly, $\dfrac1t\log([\mathcal{C}_1,\mathcal{C}_0]) = \Phi(\mathcal{C}_1) - \Phi(\mathcal{C}_0)$, it follows that $\Theta(\mathcal{C}) = \dfrac1t\omega + \Phi(\mathcal{C})$.

\appendix
\section{Calculus in the presence of an integrable distribution}\label{appendix: distributions}

In this section  we briefly review basic facts regarding differential calculus in the presence of an integrable complex distribution. We refer the reader to \cite{Kostant}, \cite{Rawnsley} and \cite{FW} for details and proofs.

For a $C^\infty$ manifold $X$ we denote by $\mathcal{O}_X$ (respectively, $\Omega^i_X$) the sheaf of \emph{complex valued} $C^\infty$ functions (respectively, differential forms of degree $i$) on $X$. Throughout this section we denote by $\mathcal{T}_X^\mathbb{R}$ the sheaf of \emph{real valued} vector fields on $X$. Let $\mathcal{T}_X := \mathcal{T}_X^\mathbb{R}\otimes_\mathbb{R}\mathbb{C}$.

\subsection{Complex distributions}
A \emph{(complex) distribution} on $X$ is a sub-bundle\footnote{A
sub-bundle is an $\mathcal{O}_X$-submodule which is a direct
summand locally on $X$.} of ${\mathcal T}_X$.

A distribution $\mathcal{P}$ is called \emph{involutive} if it is closed under the Lie bracket, i.e. $[\mathcal{P},\mathcal{P}] \subseteq \mathcal{P}$.

For a distribution $\mathcal{P}$ on $X$ we denote by $\mathcal{P}^\perp \subseteq \Omega^1_X$ the annihilator of $\mathcal{P}$ (with respect to the canonical duality pairing).

A distribution $\mathcal{P}$ of rank $r$ on $X$ is called \emph{integrable} if, locally on $X$, there exist functions $f_1,\ldots , f_r\in\mathcal{O}_X$ such that $df_1,\ldots , df_r$ form a local frame for $\mathcal{P}^\perp$.

It is easy to see that an integrable distribution is involutive. The converse is true when $\mathcal{P}$ is \emph{real}, i.e. $\overline{\mathcal{P}} = \mathcal{P}$ (Frobenius) and when $\mathcal{P}$ is a \emph{complex structure}, i.e. $\overline{\mathcal{P}} \cap \mathcal{P} =0$ and $\overline{\mathcal{P}} \oplus \mathcal{P}
= \mathcal{T}_X$ (Newlander-Nirenberg). More generally, according to Theorem 1 of \cite{Rawnsley}, a sufficient condition for integrability of a complex distribution $\mathcal{P}$ is
\begin{equation}\label{sufficient condition}
\text{$\mathcal{P}\cap\overline{\mathcal{P}}$ is a sub-bundle and both $\mathcal{P}$ and $\mathcal{P} + \overline{\mathcal{P}}$ are involutive.}
\end{equation}

\subsection{The Hodge filtration}\label{subsection: the hodge filtration}
Suppose that $\mathcal{P}$ is an involutive distribution on $X$.

Let $F_\bullet\Omega^\bullet_X$ denote the filtration by the powers of the differential ideal generated by $\mathcal{P}^\perp$, i.e. $F_{-i}\Omega^j_X = \bigwedge^i{\mathcal
P}^\perp\wedge\Omega^{j-i}_X\subseteq\Omega^j_X$. Let $\dbar$ denote the differential in $Gr^F_\bullet\Omega^\bullet_X$. The wedge product of
differential forms induces a structure of a commutative differential-graded algebra (DGA) on
$(Gr^F_\bullet\Omega^\bullet_X,\dbar)$.

In particular, $Gr^F_0\mathcal{O}_X = \mathcal{O}_X$, $Gr^F_0\Omega^1_X = \Omega^1_X/\mathcal{P}^\perp$ and $\dbar\colon \mathcal{O}_X \to Gr^F_0\Omega^1_X$ is equal to the composition $\mathcal{O}_X \xrightarrow{d} \Omega^1_X \to \Omega^1_X/\mathcal{P}^\perp$. Let $\mathcal{O}_{X/\mathcal{P}} := \ker(\mathcal{O}_X \xrightarrow{\dbar} Gr^F_0\Omega^1_X)$. Equivalently, $\mathcal{O}_{X/\mathcal{P}} = (\mathcal{O}_X)^\mathcal{P} \subset \mathcal{O}_X$, the subsheaf of functions locally constant along $\mathcal{P}$. Note that $\dbar$ is $\mathcal{O}_{X/\mathcal{P}}$-linear.

Theorem 2 of \cite{Rawnsley} says that, if $\mathcal{P}$ satisfies the condition \eqref{sufficient condition}, the higher $\dbar$-cohomology of $\mathcal{O}_X$ vanishes, i.e.
\begin{equation}\label{d-bar cohomology of functions}
H^i(Gr^F_0\Omega^\bullet_X,\dbar) = \left\{
\begin{array}{ll}
\mathcal{O}_{X/\mathcal{P}} & \text{if $i=0$} \\
0 & \text{otherwise.}
\end{array}
\right.
\end{equation}

\medskip

\noindent
In what follows we will assume that the complex distribution $\mathcal{P}$ under consideration is integrable and satisfies \eqref{d-bar cohomology of functions}. The latter is implied by the condition \eqref{sufficient condition}.

\subsection{$\dbar$-operators} Suppose that $\mathcal{E}$ is a vector bundle on $X$, i.e. a locally free $\mathcal{O}_X$-module of finite rank. A \emph{connection along $\mathcal{P}$} on
$\mathcal{E}$ is, by definition, a map $\nabla^\mathcal{P}:
\mathcal{E}\to\Omega^1_X/\mathcal{P}^\perp\otimes_{{\mathcal
O}_X}\mathcal{E}$ which satisfies the Leibniz rule
$\nabla^\mathcal{P}(fe)=f\nabla^\mathcal{P}(e)+\overline\partial
f\cdot e$ for $e\in \mathcal{E}$ and $f \in \mathcal{O}_X$. Equivalently, a connection along $\mathcal{P}$ is an
$\mathcal{O}_X$-linear map $\nabla^\mathcal{P}_{(\bullet)}\colon \mathcal{P} \to \shEnd_\mathbb{C}(\mathcal{E})$ which satisfies the Leibniz rule $\nabla^\mathcal{P}_\xi(fe)=
f\nabla^\mathcal{P}_\xi(e)+\overline\partial f\cdot e$ for $e\in \mathcal{E}$ and $f \in \mathcal{O}_X$. In particular, $\nabla^\mathcal{P}_\xi$ is $\mathcal{O}_{X/\mathcal{P}}$-linear. The two avatars of a connection along $\mathcal{P}$
are related by $\nabla^\mathcal{P}_\xi(e)=\iota_\xi\nabla^\mathcal{P}(e)$.

A connection along $\mathcal{P}$ on $\mathcal{E}$ is called \emph{flat} if the corresponding map $\nabla^\mathcal{P}_{(\bullet)}\colon \mathcal{P} \to \shEnd_\mathbb{C}(\mathcal{E})$ is a morphism of Lie algebras. We will refer to a flat connection along $\mathcal{P}$ on
$\mathcal{E}$ as a $\dbar$-operator on $\mathcal{E}$.

A connection on $\mathcal{E}$ along $\mathcal{P}$ extends uniquely to a derivation $\dbar_\mathcal{E}$ of the graded $Gr^F_0\Omega^\bullet_X$-module $Gr^F_0\Omega^\bullet_X\otimes_{\mathcal{O}_X}\mathcal{E}$ which is a $\dbar$-operator if and only if $\dbar_{\mathcal E}^2=0$. The complex $(Gr^F_0\Omega^\bullet_X\otimes_{{\mathcal O}_X}\mathcal{E}, \dbar_\mathcal{E})$ is referred to as the (corresponding) $\dbar$-complex. Since $\dbar_\mathcal{E}$ is $\mathcal{O}_{X/\mathcal{P}}$-linear, the sheaves $H^i(Gr^F_0\Omega^\bullet_X\otimes_{{\mathcal O}_X}{\mathcal E},\dbar_\mathcal{E})$ are $\mathcal{O}_{X/\mathcal{P}}$-modules. The vanishing of higher $\dbar$-cohomology of $\mathcal{O}_X$ \eqref{d-bar cohomology of functions} generalizes easily to vector bundles.

\lemma\label{dbar lemma}
Suppose that $\mathcal{E}$ is a vector bundle and $\dbar_\mathcal{E}$ is a $\dbar$-operator on $\mathcal{E}$. Then, $H^i(Gr^F_0\Omega^\bullet_X \otimes_{\mathcal{O}_X}\mathcal{E}, \dbar_\mathcal
{E})=0$ for $i\neq 0$, i.e. the $\dbar$-complex is a resolution of $\ker(\dbar_\mathcal{E})$. Moreover, $\ker(\dbar_\mathcal{E})$ is locally free over $\mathcal{O}_{X/\mathcal{P}}$ of rank $\rk_{\mathcal{O}_X}\mathcal{E}$ and the map $\mathcal{O}_X \otimes_{\mathcal{O}_{X/\mathcal{P}}} \ker(\dbar_\mathcal{E}) \to \mathcal{E}$ (the $\mathcal{O}_X$-linear extension of the
inclusion $\ker(\dbar_\mathcal{E}) \hookrightarrow \mathcal{E}$) is an isomorphism.
\endlemma

\begin{remark}
Suppose that $\mathcal{F}$ is a locally free $\mathcal{O}_{X/\mathcal{P}}$-module of finite rank. Then, $\mathcal{O}_X\otimes_{\mathcal{O}_{X/\mathcal{P}}}\mathcal{F}$ is a locally free $\mathcal{O}_X$-module of rank $\rk_{\mathcal{O}_{X/\mathcal{P}}}\mathcal{F}$ and is endowed in a canonical way with a $\dbar$-operator, namely, $\dbar\otimes\id$.
The assignments $\mathcal{F}\mapsto(\mathcal{O}_X\otimes_{\mathcal{O}_{X/\mathcal{P}}}\mathcal{F},\dbar\otimes\id)$ and $(\mathcal{E},\dbar_\mathcal{E}) \mapsto \ker(\dbar_\mathcal{E})$ are mutually inverse equivalences of suitably defined categories.
\end{remark}

\subsection{Calculus}\label{subsection: calculus}
The adjoint action of $\mathcal{P}$ on $\mathcal{T}_X$ preserves $\mathcal{P}$, hence descends to an action on $\mathcal{T}_X/\mathcal{P}$. The latter action defines a connection along $\mathcal{P}$, i.e. a canonical $\dbar$-operator on $\mathcal{T}_X/\mathcal{P}$ which is easily seen to coincide with the one induced via the duality pairing between the latter and $\mathcal{P}^\perp$.\footnote{In the case of a real polarization this connection is known as the Bott connection.} Let $\mathcal{T}_{X/\mathcal{P}} := (\mathcal{T}_X/\mathcal{P})^\mathcal{P}$ (the subsheaf of $\mathcal{P}$ invariant section, equivalently, the kernel of the $\dbar$-operator on $\mathcal{T}_X/\mathcal{P}$). The Lie bracket on $\mathcal{T}_X$ (respectively, the action of $\mathcal{T}_X$ on $\mathcal{O}_X$) induces a Lie bracket on $\mathcal{T}_{X/\mathcal{P}}$ (respectively, an action of $\mathcal{T}_{X/\mathcal{P}}$ on $\mathcal{O}_{X/\mathcal{P}}$). The bracket and the action on $\mathcal{O}_{X/\mathcal{P}}$ endow $\mathcal{T}_{X/\mathcal{P}}$ with a structure of an $\mathcal{O}_{X/\mathcal{P}}$-Lie algebroid.

The action of $\mathcal{P}$ on $\Omega^1_X$ by Lie derivative restricts to a flat connection along $\mathcal{P}$, i.e. a canonical $\dbar$-operator on $\mathcal{P}^\perp$ and, therefore, on $\bigwedge^i\mathcal{P}^\perp$ for all $i$. It is easy to see that the multiplication map $Gr^F_0\Omega^\bullet\otimes\bigwedge^i\mathcal{P}^\perp \to Gr^F_{-i}\Omega^\bullet[i]$ is an isomorphism which identifies the $\dbar$-complex of $\bigwedge^i\mathcal{P}^\perp$ with $Gr^F_{-i}\Omega^\bullet[i]$. Let $\Omega^i_{X/\mathcal{P}}:= H^i(Gr^F_{-i}\Omega^\bullet_X,\dbar)$ (so that $\mathcal{O}_{X/\mathcal{P}}:=\Omega^0_{X/\mathcal{P}}$). Then,  $\Omega^i_{X/\mathcal{P}} \subset \bigwedge^i\mathcal{P}^\perp \subset \Omega^i_X$. The wedge product of differential forms
induces a structure of a graded-commutative algebra  on $\Omega^\bullet_{X/\mathcal {P}} := \oplus_i \Omega^i_{X/\mathcal{P}}[-i] = H^\bullet(Gr^F\Omega^\bullet_X,\dbar)$. The multiplication induces an isomorphism $\bigwedge^i_{\mathcal{O}_{X/\mathcal{P}}} \Omega^1_{X/\mathcal{P}} \to \Omega^i_{X/\mathcal{P}}$. The de Rham differential $d$ restricts to the map $d :
\Omega^i_{X/\mathcal{P}} \to \Omega^{i+1}_{X/\mathcal{P}}$ and the complex $\Omega^\bullet_{X/\mathcal{P}}:=(\Omega^\bullet_{X/\mathcal{P}},d)$ is a commutative DGA.

The Hodge filtration $F_\bullet\Omega^\bullet_{X/\mathcal{P}}$ is defined by
\[
F_i\Omega^\bullet_{X/\mathcal{P}} = \oplus_{j\geq -i} \Omega^j_{X/\mathcal{P}} ,
\]
so that the inclusion $\Omega^\bullet_{X/\mathcal{P}} \hookrightarrow \Omega^\bullet_X$ is filtered with respect to the Hodge filtration. It follows from Lemma \ref{dbar lemma} that it is, in fact, a filtered quasi-isomorphism.

The duality pairing $\mathcal{T}_X/\mathcal{P}\otimes\mathcal{P}^\perp \to \mathcal{O}_X$ restricts to a non-degenerate pairing $\mathcal{T}_{X/\mathcal{P}} \otimes_{\mathcal{O}_{X/\mathcal{P}}} \Omega^1_{X/\mathcal{P}} \to \mathcal{O}_{X/\mathcal{P}}$.
The action of $\mathcal{T}_X/\mathcal{P}$ on $\mathcal{O}_{X/\mathcal{P}}$ the pairing and the de Rham differential are related by the usual formula $\xi(f)=\iota_\xi df$, for $\xi \in \mathcal{T}_{X/\mathcal{P}}$ and $f\in\mathcal{O}_{X/\mathcal{P}}$.

\subsection{Symplectic geometry}
A Poisson bracket on $\mathcal{O}_{X/\mathcal{P}}$ is a bi-derivation
\[
\{\cdot,\cdot\}\colon \mathcal{O}_{X/\mathcal{P}}\otimes\mathcal{O}_{X/\mathcal{P}} \to \mathcal{O}_{X/\mathcal{P}}
\]
which satisfies the Jacobi identity, i.e. defines a $\mathbb{C}$-Lie algebra structure on $\mathcal{O}_{X/\mathcal{P}}$. A Poisson bracket corresponds to a Poisson bi-vector $\pi\in \Gamma(X;\bigwedge^2\mathcal{T}_{X/\mathcal{P}})$ by the formula $\{f,g\} = \pi(df,dg)$. If the Poisson bi-vector $\pi$ is non-degenerate, then $\omega = \pi^{-1}$ is a closed non-degenerate 2-form usually referred to as the symplectic form, and the Poisson bracket is said to be symplectic. Clearly, existence of a symplectic Poisson bracket implies that $\rk_{\mathcal{O}_{X/\mathcal{P}}}\mathcal{T}_{X/\mathcal{P}}$ is even.

Suppose that $\{\cdot,\cdot\}$ is a symplectic Poisson bracket. Let $\rk_{\mathcal{O}_{X/\mathcal{P}}}\mathcal{T}_{X/\mathcal{P}} = 2n$. Darboux Lemma says that every point $x\in X$ has a neighborhood $x\in U \subset X$ such that there exist functions $p_i, q_i \in \mathcal{O}_{X/\mathcal{P}}(U)$, $i=1,\ldots, n$, which satisfy the \emph{canonical relations} $\{p_i,p_j\} = \{q_i,q_j\} = 0$, $\{p_i,q_j\} = \delta_{ij}$.

\section{Torsors and gerbes}\label{appendix: algebroid stacks}
In what follows we will be considering gerbes with abelian lien. Below we briefly recall some relevant notions and constructions with the purpose of establishing notations. We refer the reader to \cite{Dpic}, \cite{Breen} and \cite{M} for detailed treatment of Picard stacks and gerbes respectively.

Suppose that $X$ is a topological space.

\subsection{Picard stacks}
We recall the definitions from \emph{1.4.Champs de Picard strictement commutatifs} of \cite{Dpic}. 

A (strictly commutative) \emph{Picard groupoid} $\mathcal{P}$ is a non-empty groupoid equipped with a functor $+ \colon \mathcal{P}\times\mathcal{P}\to\mathcal{P}$ and functorial isomorphisms
\begin{itemize}
\item $\sigma_{x,y,z} \colon (x+y)+z \to x+(y+z)$
\item $\tau_{x,y} \colon x+y \to y+x$
\end{itemize}
rendering $+$ associative and strictly commutative, and such that for each object $x\in\mathcal{P}$ the functor $y \mapsto x+y$ is an equivalence.

A \emph{Picard stack} on $X$ is a stack in groupoids $\mathcal{P}$ equipped with a functor
$+ \colon \mathcal{P}\times\mathcal{P}\to\mathcal{P}$ and functorial isomorphisms $\sigma$ and $\tau$ as above, which, for each open subset $U \subseteq X$, endow the category $\mathcal{P}(U)$ a structure of a Picard groupoid.

\subsection{Torsors}
Suppose that $A$ is a sheaf of abelian groups on $X$. The stack of $A$-torsors will be denoted by $A[1]$; it is a gerbe since all $A$-torsors are locally trivial.

Suppose that $\phi \colon A \to B$ is a morphism of sheaves of abelian groups. The assignment $A[1] \ni T \mapsto \phi T := T\times_A B \in B[1]$ extends to a morphism $\phi \colon A[1] \to B[1]$ of stacks. There is a canonical map of sheaves of torsors $\phi = \phi_T \colon T \to \phi T$ compatible with the map $\phi$ of abelian groups and respective actions.

Suppose that $A$ and $B$ are sheaves of abelian groups. The assignment $A[1]\times B[1] \ni (S,T) \mapsto S\times T \in (A\times B)[1]$ extends to a morphism of stacks $\times \colon A[1]\times B[1] \to (A\times B)[1]$.

Suppose that $A$ is a sheaf of abelian groups with the group structure $+ \colon A\times A \to A$. The latter is a morphism of sheaves of groups since $A$ is abelian. The assignment $A[1]\times A[1] \ni (S,T) \mapsto S+T := +(S\times T)$ defines a structure of a Picard stack on $A[1]$. If $\phi \colon A \to B$ is a morphism of sheaves of abelian groups the corresponding morphism $\phi \colon A[1] \to B[1]$ is a morphism of Picard stacks.

As a consequence, the set $\pi_0A[1](X)$ of isomorphism classes of $A$-torsors is endowed with a canonical structure of an abelian group. There is a canonical isomorphism of groups $\pi_0A[1](X) \cong H^1(X;A)$.

\subsection{Gerbes}
A gerbe on $X$ is a stack in groupoids which is locally non-empty and locally connected. For a sheaf $A$ of abelian groups a $A$-gerbe $\mathcal{S}$ is a gerbe together with functorial isomorphisms $A(U) \cong \Aut(s)$, where $U \subseteq X$ is an open subset such that $\mathcal{S}(U)\neq\varnothing$ and $s\in\mathcal{S}(U)$. Morphisms of $A$-gerbes are required to respect the above identifications.

Thus, a $A$-gerbe is a twisted form of (i.e. locally equivalent to) $A[1]$. The 2-stack of $A$-gerbes will be denoted $A[2]$. Since all $A$-gerbes are locally equivalent the 2-stack $A[2]$ is a 2-gerbe.

Via the equivalence $\shEq_A(A[1],\mathcal{S}) \cong \mathcal{S}$ every $A$-gerbe $\mathcal{S}$ (is equivalent to one which) admits a canonical action of the Picard stack $A[1]$ by autoequivalences denoted $+ \colon A[1]\times \mathcal{S} \to \mathcal{S}$, $(T, L) \mapsto T+L$ endowing $\mathcal{S}$ with a structure of a 2-torsor under $A[1]$. We shall not make distinction between $A$-gerbes and  2-torsors under $A[1]$ and use the notation $A[2]$ for both.

Suppose that $\phi \colon A \to B$ is a morphism of sheaves of abelian groups and $\mathcal{S}$ is an $A$-gerbe. In particular, for any two (locally defined) objects $s_1, s_2 \in \mathcal{S}$ the sheaf $\shHom_\mathcal{S}(s_1, s_2)$ is an $A$-torsor. The stack $\phi\mathcal{S}$ is defined as the stack associated to the prestack with the same objects as $\mathcal{S}$ and $\shHom_{\phi\mathcal{S}}(s_1, s_2) := \phi\shHom_\mathcal{S}(s_1, s_2)$. Then, $\phi\mathcal{S}$ is a $B$-gerbe and the assignment $\mathcal{S} \mapsto \phi\mathcal{S}$ extends to a morphism $\phi \colon A[2] \to B[2]$. There is a canonical morphism of stacks $\phi = \phi_\mathcal{S} \colon \mathcal{S} \to \phi\mathcal{S}$ which induces the map $\phi \colon A \to B$ on groups of automorphisms.

The Picard structure on $A[1]$ gives rise to one on $A[2]$ defined in analogous fashion. As a consequence, the set $\pi_0A[2](X)$ of equivalence classes of $A[1]$-torsors is endowed with a canonical structure of an abelian group. There is a canonical isomorphism of groups $\pi_0A[2](X) \cong H^2(X;A)$.

\subsection{Picard stacks and complexes}\label{subsection: Picard stacks and complexes}
Let $A^0 \xrightarrow{d} A^1$ be a complex of sheaves of abelian groups on $X$ concentrated in degrees zero and one. Recall that a $(A^0 \xrightarrow{d} A^1)$-torsor is a pair $(T, \tau)$, where $T$ is a $A^0$-torsor and $\tau$ is a trivialization (i.e. a section) of the $A^1$-torsor $d(T) = T\times_{A^0} A^1$. A morphism of $(A^0 \xrightarrow{d} A^1)$-torsors $\phi \colon (S,\sigma) \to (T,\tau)$ is a morphism of $A^0$-torsors $\phi \colon S \to T$ such that the induced morphism of $A^1$-torsors $d(\phi) \colon dS \to dT$ which commutes with respective trivializations, i.e. $d(\phi)(\sigma) = \tau$. Alternatively, a $(A^0 \xrightarrow{d} A^1)$-torsor is a pair $(T, \mathtt{c})$, where $T$ is a $A^0$-torsor and $\mathtt{c}\colon T \to A^1$ is a map of sheaves which satisfies $\mathtt{c}(t+a) = \mathtt{c}(t)+da$. The latter is obtained from a trivialization as the composition $T \xrightarrow{t\mapsto (t,0)} T\times_{A^0} A^1 \cong A^1$. Conversely, $\mathtt{c}$ extends  canonically to the morphism of $A^1$-torsors $d(T) = T\times_{A^0} A^1 \to A^1$.

The monoidal structure on the category of $(A^0 \xrightarrow{d} A^1)$-torsors is defined as follows. Suppose that $(S,\sigma)$ and $(T,\tau)$ are $(A^0 \xrightarrow{d} A^1)$-torsors. The sum $(S,\sigma)+(T,\tau)$ is represented by $(S+T, \sigma + \tau)$, where $\sigma + \tau$ is the trivialization of $d(S)+d(T) = d(S+T)$ induced by $\sigma$ and $\tau$.

Locally defined $(A^0 \xrightarrow{d} A^1)$-torsors form a Picard stack on $X$ which we will denote by $(A^0 \xrightarrow{d} A^1)[1]$. By a result of P.~Deligne (\cite{Dpic}, Proposition 1.4.15) all Picard stacks arise in this way. The group (under the operation induced by the monoidal structure) of isomorphism classes of $(A^0 \xrightarrow{d} A^1)$-torsors on $X$, i.e. $\pi_0(A^0 \xrightarrow{d} A^1)[1](X)$ is canonically isomorphic to $H^1(X;A^0 \xrightarrow{d} A^1)$.

%Let $(A^0 \xrightarrow{d} A^1)^+$ denote the sheaf of Picard groupoids with $(A^0 \xrightarrow{d} A^1)^+(U)$, $U$ and open subset of $X$, the groupoid associated to the action of $A^0$ on $A^1$ via $d$. In other words, the set of objects of $(A^0 \xrightarrow{d} A^1)^+(U)$ is $A^1(U)$ and the set of morphisms is $A^0(U)\times A^1(U)$, where $(a, b)$ is a morphism $b \to b+da$. The composition is given by $(a,b)\circ(c,d) = (a+c,d)$. The monoidal structure is given by the addition. The canonical morphism of prestacks in Picard groupoids given by
%\[
%(A^0 \xrightarrow{d} A^1)^+ \to (A^0 \xrightarrow{d} A^1)[1] \colon A^1(U)\ni b \mapsto (A^0\vert_U, d+b)
%\]
%induces an equivalence of associated stacks.

A morphism of complexes
\[
\phi = (\phi^0,\phi^1) \colon (A^0 \xrightarrow{d_A} A^1) \to (B^0 \xrightarrow{d_B} B^1)
\]
induces the morphism of Picard stacks $\phi \colon (A^0 \xrightarrow{d_A} A^1)[1] \to (B^0 \xrightarrow{d_B} B^1)[1]$ defined by $\phi(S,\sigma) = (\phi^0S,\phi\sigma)$, where $\phi\sigma \colon \phi^0S \to B^1$ is the unique trivialization of $d_B\phi^0S$ such that the composition $S \to \phi^0S \xrightarrow{\phi\sigma} B^1$ is equal to $\phi^1\circ\sigma$.

\subsection{2-torsors}\label{subsection: 2-torsors}
A (2-)torsor under the Picard stack $\mathcal{P}$ is a stack $\mathcal{S}$ endowed with an action $+\colon\mathcal{P}\times\mathcal{S} \to \mathcal{S}$ which is locally equivalent as $\mathcal{P}$ equipped with the action of $\mathcal{P}$ by translations. The 2-stack of $\mathcal{P}$-torsors will be denoted by $\mathcal{P}[1]$. Torsors under the Picard stack $(A^0 \xrightarrow{d} A^1)[1]$ admit an alternative description as $(A^0 \xrightarrow{d} A^1)$-gerbes. 

A $(A^0 \xrightarrow{d} A^1)$-gerbe is equivalent to the data $(\mathcal{S}, \tau)$, where $\mathcal{S}$ is an $A^0$-gerbe and $\tau$ is a trivialization of the $A^1$-gerbe $d\mathcal{S}$, i.e. an equivalence $\tau \colon d\mathcal{S} \to A^1[1]$. The composition $\mathcal{S} \xrightarrow{d} d\mathcal{S} \xrightarrow{\tau} A^1[1]$ is a functorial assignment of an $A^1$-torsor $\tau(s)$ to a (locally defined) object $s\in\mathcal{S}$. 

A 1-morphism of $(A^0 \xrightarrow{d} A^1)$-gerbes $(\mathcal{S}, \tau_\mathcal{S}) \to (\mathcal{T}, \tau_\mathcal{T})$ is a pair $(F, \lambda)$, where $F\colon \mathcal{S} \to \mathcal{T}$ is a 1-morphism of $A^0$-gerbes and $\lambda$ is a 2-morphism $\tau_\mathcal{S} \to \tau_\mathcal{T}\circ dF$.

A 2-morphism $\eta \colon (F, \lambda_F) \to (G, \lambda_G)$ between 1-morphisms $(\mathcal{S}, \tau_\mathcal{S}) \to (\mathcal{T}, \tau_\mathcal{T})$ is a 2-morphism $\eta \colon F \to G$ which satisfies $\lambda_G = (\id_{\tau_\mathcal{T}}\diamond d\eta)\circ\lambda_F$.

The 2-stack of $(A^0 \xrightarrow{d} A^1)$-gerbes will be denoted $(A^0 \xrightarrow{d} A^1)[2]$. 

Every $(A^0 \xrightarrow{d} A^1)$-gerbe admits a canonical action of $(A^0 \xrightarrow{d} A^1)[1]$ by autoequivalences as follows. The action of $A^0[1]$ on $\mathcal{S}$ extends to an action of $(A^0 \xrightarrow{d} A^1)[1]$ on $(\mathcal{S}, \tau)$. Namely, $(T,\lambda) \in (A^0 \xrightarrow{d} A^1)[1]$ gives rise to the autoequivalence $(T+(\cdot), \lambda)$.

The Picard structure on $(A^0 \xrightarrow{d} A^1)[1]$ gives rise to one on $(A^0 \xrightarrow{d} A^1)[2]$ defined in an analogous fashion. As a consequence, the set $\pi_0(A^0 \xrightarrow{d} A^1)[2](X)$ of equivalence classes of $(A^0 \xrightarrow{d} A^1)$-gerbes is endowed with a canonical structure of an abelian group. There is a canonical isomorphism of groups $\pi_0(A^0 \xrightarrow{d} A^1)[2](X) \cong H^2(X;A^0 \xrightarrow{d} A^1)$.

\subsection{Extensions and torsors}\label{subsection: Extensions and torsors}
Suppose that $A$ and $B$ are sheaves of modules over a sheaf of rings $k$ on $X$. We denote by $\gpdExt_k^1(A,B)$ the category of extensions
\[
E:\ \ \ 0 \to B \to E \to A \to 0
\]
The morphisms in $\gpdExt_k^1(A,B)$ are maps of short exact sequences which induce the identity maps on $A$ and $B$. As any such map is an isomorphism, the category $\gpdExt_k^1(A,B)$ is a groupoid. The Baer sum of extensions endows the category $\gpdExt_k^1(A,B)$ with a structure of a Picard groupoid. A zero object is given by the split extension $E_\circ = B\oplus A$.

The assignment $X \supseteq U \mapsto \gpdExt_k^1(A\vert_U,B\vert_U)$, $U$ an open subset, defines a Picard stack on $X$ denoted $\shgpdExt_k^1(A,B)$. For any injective resolution $I^0 \to I^1 \to \cdots$ of $B$ there is an equivalence of Picard stacks $(\tau^{\leqslant 1}\shHom_k(A,I^\bullet))[1] \cong \shgpdExt_k^1(A,B)$ and, in particular, $\pi_0\shgpdExt_k^1(A,B) = \shExt_k^1(A,B)$ and $\pi_1\shgpdExt_k^1(A,B) = \shHom_k(A,B)$, where $\phi\in\shHom_k(A,B)$ acts by $(b,a) \mapsto (b+\phi(a))$.

The canonical functor $\shHom_k(A,B)[1] \to \shgpdExt_k^1(A,B)$ defined by $T\mapsto T\times_{\Hom_k(A,B)}E_\circ$ is fully faithful with essential image the subcategory of extensions which admit a splitting locally on $X$. Conversely, suppose that the extension $E$ is split locally on X. Then, the map $\shHom_k(A, E) \to \shHom_k(A, A)$ is an epimorphism and $E_{/\id_A} := \shHom_k(A, E)\times_{\shHom_k(A, A)} \id_A$, where $\id_A \colon \ast_X \to \shHom_k(A, A)$ is the global section $\id_A\in\Hom_k(A, A)$, is the corresponding $\shHom_k(A,B)$-torsor.

In particular, using the canonical identification $\shHom_k(k,B) = B$, we obtain the canonical equivalence $B[1] \to \shgpdExt_k^1(k,B)$. A quasi-inverse associates to $T\in B[1]$ the unique extension
\[
0 \to B \to \widetilde{T} \to k \to 0
\]
such that $T = \ast_X\times_k\widetilde{T}$, where $\ast_X \xrightarrow{1} k$ is the unit section.

\subsection{Extensions and 2-torsors}\label{subsection: Extensions and 2-torsors}
We continue with notations introduced in \ref{subsection: Extensions and torsors}.
%In what follows we shall assume that all extensions a split locally on $X$.
An extension
\[
\widetilde{E} \colon \ \ \ 0 \to C \to \widetilde{E} \to E \to 0 \ \ \in \gpdExt_k^1(E,C)
\]
gives rise to the extension $\widetilde{E}_{/B} := B\times_E\widetilde{E} \in \gpdExt_k^1(B,C)$. The assignment $\widetilde{E} \mapsto \widetilde{E}_{/B}$ extends to a morphism of Picard stacks $\shgpdExt_k^1(E,C) \to \shgpdExt_k^1(B,C)$; for $K \in \gpdExt_k^1(B,C)$ we denote by $\shgpdExt_k^1(E,C)_{/K}$ the corresponding fiber. The Baer sum operation on $\shgpdExt_k^1(E,C)$ restricts to the pairings
\[
\shgpdExt_k^1(E,C)_{/K}\times\shgpdExt_k^1(E,C)_{/L}\to\shgpdExt_k^1(E,C)_{/K\dotplus L}
\]
and, in particular, a structure of a Picard stack on $\shgpdExt_k^1(E,C)_{/\mathbf{0}}$, where $\mathbf{0}$ is the split extension, and to a structure of a $\shgpdExt_k^1(E,C)_{/\mathbf{0}}$-torsor on $\shgpdExt_k^1(E,C)_{/K}$.

The morphism of Picard stacks
\[
\shgpdExt_k^1(A,C) \to \gpdExt_k^1(E,C)\ \colon F \mapsto E\times_A F
\]
establishes the canonical equivalence $\shgpdExt_k^1(A,C) \cong \shgpdExt_k^1(E,C)_{/\mathbf{0}}$. A quasi-inverse is given by $\widetilde{E}\mapsto \coker(B\to\widetilde{E})$, where the map $B\to\widetilde{E}$ is deduced form the splitting of $\widetilde{E}_{/B}$.  In what follows we shall regard $\shgpdExt_k^1(E,C)_{/K}$ as a torsor under $\shgpdExt_k^1(A,C)$. 

We apply the above considerations to an exact sequence
\begin{equation}\label{extension of length two}
0 \to C \to E^0 \to E^1 \to A \to 0
\end{equation}
with $B = \coker(C \to E^0) = \ker(E^1 \to A)$, $E=E^1$ and $K=E^0$ to obtain the $\shgpdExt_k^1(A,C)$-torsor $\shgpdExt_k^1(E^1,C)_{/E^0}$. The latter is equivalent to the category of commutative diagrams
\[
\begin{CD}
& & 0 & & 0 \\
& & @VVV @VVV \\
0 @>>> C @= C @>>> 0\\
& & @VVV @VVV @VVV\\
0 @>>> E^0 @>>> \widetilde{E^1} @>>> A @>>> 0 \\
& & @VVV @VVV @| \\
0 @>>> B @>>> E^1 @>>> A @>>> 0 \\
& & @VVV @VVV @VVV \\
& & 0 & & 0 & & 0
\end{CD}
\]
with exact rows and columns and morphisms thereof which induced identity maps on all objects except $\widetilde{E^1}$.

Suppose that
\[
\begin{CD}
0 @>>> C @>>> E^0 @>>> E^1 @>>> A @>>> 0 \\
& & @VVV @V{\phi^0}VV @VV{\phi^1}V @| \\
0 @>>> D @>>> F^0 @>>> F^1 @>>> A @>>> 0
\end{CD}
\]
is a commutative diagram with exact rows. Let $K := \coker(D \to F^0) = \ker(F^1 \to A)$. Let $\phi_*\widetilde{E^1} := F^0\bigsqcup_{E^0}\widetilde{E^1}$. The composition
\[
\phi_*\widetilde{E^1} = F^0\bigsqcup_{E^0}\widetilde{E^1} \to K\bigsqcup_{F^0}F^0\bigsqcup_{E^0}\widetilde{E^1} \cong K\bigsqcup_{B}B\bigsqcup_{E^0}\widetilde{E^1} \cong K\bigsqcup_{B}E^1 \to F^1
\]
gives rise to the exact sequence
\[
0 \to D \to \phi_*\widetilde{E^1} \to F^1 \to 0 .
\]
Moreover, the canonical map $F^0 \to K\times_{F^1}\phi_*\widetilde{E^1}$ is an isomorphism, hence $\phi_*\widetilde{E^1} \in \shgpdExt_k^1(F^1,D)_{/F^0}$. The assignment $\widetilde{E^1} \mapsto \phi_*\widetilde{E^1}$ extends to a morphism
\[
\phi_* \colon \shgpdExt_k^1(E^1,C)_{/E^0} \to \shgpdExt_k^1(F^1,D)_{/F^0} .
\]
of torsors relative to the morphism of Picard stacks $\shgpdExt_k^1(A,C) \to \shgpdExt_k^1(A,D)$ induced by the map $\phi^0\vert_C \colon C \to D$.

Suppose that $A=k$. Then, $\shgpdExt_k^1(A,C) \cong C[1]$, and the class of the $C[1]$-torsor $\shgpdExt_k^1(E^1,C)_{/E^0}$ in $H^2(X;C)$ is equal to that of the extension \eqref{extension of length two}.

\subsection{2-gerbes}
For a Picard stack $\mathcal{P}$ a $\mathcal{P}$-gerbe is a 2-stack in 2-groupoids $\mathfrak{S}$ which is locally non-empty and locally connected together with the data of
equivalences $\eta_s \colon \mathcal{P}\vert_U \to \shAut_\mathfrak{S}(s)$, where $U\subseteq X$ is an open subset such that $\mathfrak{S} \neq \varnothing$ and $s\in\mathfrak{S}(U)$ which are coherent in the sense specified in \cite{Breen}. The collection of $\mathcal{P}$-gerbes naturally forms a $3$-category and the associated $3$-stack is denoted by $\mathcal{P}[2]$. A $\mathcal{P}$-gerbe $\mathfrak{G}$ naturally defines a torsor under $\mathcal{P}[1]$, namely $\shEq(\mathcal{P}[1], \mathfrak{S})$. If $\mathcal{P} = (A^{0} \xrightarrow{d} A^{1})[1]$, then the common value of $(A^{0} \xrightarrow{d} A^{1})[1][2] = (A^{0} \xrightarrow{d} A^{1})[2][1]$ is denoted by $(A^{0} \xrightarrow{d} A^{1})[3]$. The $3$-stack $(A^{0} \xrightarrow{d} A^{1})[3]$ admits a natural Picard structure so that $\pi_{0}(A^{0} \xrightarrow{d} A^{1})[3](X)$ is an abelian group which is canonically isomorphic to $H^{3}(X, A^{0} \xrightarrow{d} A^{1})$.

%Let $\mathcal{P} = (A^{0} \xrightarrow{d} A^{1})[1]$ and $\mathcal{Q} = (B^{0} \xrightarrow{d} B^{1})[1]$. 

For a map of complexes morphism of complexes
\[
\phi = (\phi^0,\phi^1) \colon (A^0 \xrightarrow{d_A} A^1) \to (B^0 \xrightarrow{d_B} B^1)
\]
and a $(A^0 \xrightarrow{d_A} A^1)[1]$-gerbe $\mathfrak{S}$, the $(B^0 \xrightarrow{d_B} B^1)[1]$-gerbe $\phi\mathfrak{S}$ is defined as the 2-stack associated to the pre-$2$-stack with same objects as $\mathfrak{S}$ and $\shHom_{\phi\mathfrak{S}}(s,t) = \phi\shHom_{\mathfrak{S}}(s,t)$.

\subsection{Algebroids}\label{subsection: algebroids}
Let $k$ a commutative ring with unit. Recall that a $k$-algebroid is a stack in $k$-linear categories $\mathcal{C}$ such that the substack of isomorphisms $i\mathcal{C}$ (which is a stack in groupoids) is a gerbe.

For a $k$-algebra $A$ we denote by $A^+$ the $k$-linear category with one object denoted by $\ast$, whose endomorphism algebra is $A$.

Suppose that $\mathcal{A}$ is a sheaf of $k$-algebras on $X$. The assignment $U\mapsto \mathcal{A}(U)^+$, $U \subseteq X$ an open subset,  defines a prestack on $X$; we denote  the associated stack by $\mathcal{A}^+$.  Note that $\mathcal{A}^+$ is equivalent to the stack of $\mathcal{A}^\op$-modules locally isomorphic to $\mathcal{A}$. Clearly, the category $\mathcal{A}^+(X)$ is non-empty.

Conversely, let $\mathcal{C}$ be a $k$-algebroid such that the category $\mathcal{C}(X)$ is non-empty. Let $L\in\mathcal{C}(X)$, and let $\mathcal{A} := \shEnd_\mathcal{C}(L)^\op$. The assignment $U \mapsto (\mathcal{A}(U)^+ \to \mathcal{C}(U) \colon \ast \mapsto L\vert_U)$ extends to an equivalence $\mathcal{A}^+ \to \mathcal{C}$.

For a sheaf of $k$-algebras $\mathcal{A}$ on $X$ a \emph{twisted form of $\mathcal{A}$} is a $k$-algebroid locally $k$-linearly equivalent to $\mathcal{A}^+$.

For a $k$-algebra $A$ we denote by $\zentrum(A)$ the center of $A$. For a $k$-linear stack $\mathcal{C}$ we denote by $\zentrum(\mathcal{C})$ the center of $\mathcal{C}$, i.e. the sheaf of $k$-algebras defined by $U \mapsto \End_k(\id_{\mathcal{C}\vert_U})$. Note that, for a sheaf of $k$-algebras $\mathcal{A}$, $\zentrum(\mathcal{A}^+) = \zentrum(\mathcal{A})$.

For a $k$-algebroid $\mathcal{C}$ there is a canonical action of $\zentrum(\mathcal{C})^\times$-torsors on $\mathcal{C}$, denoted $L \mapsto T\otimes L$, for $L \in \mathcal{C}$, $T \in \zentrum(\mathcal{C})^\times[1]$, hence a canonical monoidal functor $\zentrum(\mathcal{C})^\times[1] \to \shAut(\mathcal{C})$.

%Let $X$ be a $C^\infty$ manifold equipped with an integrable complex distribution $\mathcal{P}$. Twisted forms of the sheaf of $\mathbb{C}$-algebras $\mathcal{O}_{X/\mathcal{P}}$ are in one-to-one correspondence with $\mathcal{O}_{X/\mathcal{P}}^\times$-gerbes via $\mathcal{S} \mapsto i\mathcal{S}$. Equivalence classes of twisted forms of $\mathcal{O}_{X/\mathcal{P}}$ form an group canonically isomorphic to $H^2(X;\mathcal{O}_{X/\mathcal{P}}^\times)$.


\begin{thebibliography}{10}
\bibitem[Bressler, Gorokhovsky, Nest \& Tsygan, 2007]{DQgerbes} P.~Bressler, A.~Gorokhovsky, R.~Nest, and B.~Tsygan, Deformation quantization of gerbes, Adv. Math. 214(1) (2007), 230--266.

\bibitem[Bressler, Gorokhovsky, Nest \& Tsygan, 2017]{QCL} P.~Bressler, A.~Gorokhovsky, R.~Nest, and B.~Tsygan, On quasi-classical limits of DQ-algebroids, Compositio Math. 153 (2017), 41--67.

\bibitem[Breen, 1994]{Breen} L.~Breen, On the classification of 2-gerbes and 2-stacks, Astérisque 225 (1994), 160 pp.

\bibitem[Deligne, 1973]{Dpic} P.~Deligne, La formule de dualite globale, Th\'eorie des topos et cohomologie \'etale des sch\'emas, Tome 3. S\'eminaire de G\'eom\'etrie Alg\'ebrique du Bois-Marie 1963–1964 (SGA 4). Lecture Notes in Mathematics, Vol. 305. Springer-Verlag, Berlin-New York, 1973, 481--587.

\bibitem[Deligne, 1995]{D} P.~Deligne, D\'eformations de l'algebre des fonctions d'une vari\'et\'e symplectique, Selecta Mathematica, New Series Vol. 1 No. 4 (1995), 667--697.

\bibitem[D'Agnolo \& Kashiwara, 2011]{andrea-masaki} A.~D'Agnolo and M.~Kashiwara, On quantization of complex symplectic manifolds, Comm. Math. Phys., 308(1) (2011), 81--113.

\bibitem[D'Agnolo \& Polesello, 2005]{dagnolo-polesello_complex-involutive-submanifolds}
A.~D'Agnolo and P.~Polesello, Deformation quantization of complex involutive submanifolds, Noncommutative geometry and physics, World Sci. Publ., Hackensack, NJ, 2005, 127--137.

\bibitem[De~Wilde \&  Lecomte, 1983]{WL1983} M.~De~Wilde and P.~B.~A. Lecomte, Existence of star-products and of formal deformations of the Poisson Lie algebra of arbitrary symplectic manifolds, Lett. Math. Phys. 7(6) (1983), 487--496.

\bibitem[De~Wilde \&  Lecomte, 1985]{WL1985} M.~De~Wilde and P.~B.~A. Lecomte, Existence of star-products on exact symplectic manifolds, Ann. Inst. Fourier (Grenoble) 35(2) (1985),117--143.

\bibitem[Fedosov, 1985]{fedosov} B.~V.~Fedosov, Formal quantization, Some problems in modern mathematics and their applications to problems in mathematical physics (Russian), Moskov. Fiz.-Tekhn. Inst., Moscow, 1985, 129--136, vi.

\bibitem[Fischer \& Williams, 1979]{FW} H.~R.~Fischer and F.~L.~Williams, Complex-foliated structures I. Cohomology of the Dolbeault-Kostant complexes, Trans. Amer. Math. Soc. 252 (1979), 163--195.

\bibitem[Kashiwara, 1996]{Kashiwara-contact} M.~Kashiwara, Quantization of contact manifolds, Publ. Res. Inst. Math. Sci., 32(1) (1996), 1--7.

\bibitem[Kontsevich, 2001]{kontsevich} M.~Kontsevich, Deformation quantization of algebraic varieties, Lett. Math. Phys., 56 (3) (2001), 271--294.

\bibitem[Kostant, 1970]{Kostant} B.~Kostant, Quantization and unitary representations. In: Taam C.T. (eds) Lectures in Modern Analysis and Applications III. Lecture Notes in Mathematics, vol 170. Springer, Berlin, Heidelberg, 1970, 87--208.

\bibitem[Kashiwara \& Schapira, 2012]{KS} M.~Kashiwara and P.~Schapira, Deformation quantization modules, Asterisque, 345 (2012), Soc. Math. France.

\bibitem[Milne, 2003]{M} J.~S.~Milne, Gerbes and abelian motives, arXiv:math/0301304 [math.AG].

\bibitem[Nest \& Tsygan, 2004]{NT} R.~Nest, B.L.~Tsygan, Remarks on modules over deformation quantization algebras, Mosc. Math. J., 4:4 (2004), 911--940.

\bibitem[Polesello, 2008]{polesello-classification} P.~Polesello, Classification of deformation quantization algebroids on complex symplectic manifolds, Publ. Res. Inst. Math. Sci., 44(3) (2008), 725--748.

\bibitem[Rawnsley, 1977]{Rawnsley} J.~H.~Rawnsley, On the cohomology groups of a polarisation and diagonal quantisation, Trans. Amer. Math. Soc., 230 (1977), 235--255.

\end{thebibliography}
\end{document}